\documentclass[10pt]{article}
\usepackage{amsmath, amssymb, amsthm}
\usepackage{graphicx, tikz}
\usepackage{authblk, setspace, cite}
\usepackage[top=1 in, bottom=1 in, left=1.2 in, right=1.2 in]{geometry}

\newtheorem{theorem}{Theorem}[section]
\newtheorem{lemma}[theorem]{Lemma}
\newtheorem{corollary}[theorem]{Corollary}

\newtheorem{definition}[theorem]{Definition}

\numberwithin{equation}{section}

\def\p{\partial}  \def\ora{\overrightarrow}
\def\ol{\overline}		\def\m{\mathbb}		
\def\O{\Omega}  \def\lam{\lambda}    
\def\t{\tilde}	\def\wt{\widetilde}
\def\be{\begin{equation}}     \def\ee{\end{equation}}

\title{Improvements on lower bounds for the blow-up time under local nonlinear Neumann conditions}

\author[a]{Xin Yang\thanks{Email: yang2x2@ucmail.uc.edu}}
\author[b]{Zhengfang Zhou\thanks{Email: zfzhou@math.msu.edu}}
\affil[a]{Department of Mathematical Sciences, University of Cincinati, Cincinnati, OH 45220, USA}
\affil[b]{Department of Mathematics, Michigan State University, East Lansing, MI 48824, USA}

\date{}

\begin{document}
\maketitle

\begin{abstract}
This paper studies the heat equation $u_t=\Delta u$ in a bounded domain $\Omega\subset\mathbb{R}^{n}(n\geq 2)$ with positive initial data and a local nonlinear Neumann boundary condition: the normal derivative $\partial u/\partial n=u^{q}$ on partial boundary $\Gamma_1\subseteq \partial\Omega$ for some $q>1$, while $\partial u/\partial n=0$ on the other part. We investigate the lower bound of the blow-up time $T^{*}$ of $u$ in several aspects. First, $T^{*}$ is proved to be at least of order $(q-1)^{-1}$ as $q\rightarrow 1^{+}$. Since the existing upper bound is of order $(q-1)^{-1}$, this result is sharp. Secondly, if $\Omega$ is convex and $|\Gamma_{1}|$ denotes the surface area of $\Gamma_{1}$, then $T^{*}$ is shown to be at least of order $|\Gamma_{1}|^{-\frac{1}{n-1}}$ for $n\geq 3$ and $|\Gamma_{1}|^{-1}\big/\ln\big(|\Gamma_{1}|^{-1}\big)$ for $n=2$ as $|\Gamma_{1}|\rightarrow 0$, while the previous result is $|\Gamma_{1}|^{-\alpha}$ for any $\alpha<\frac{1}{n-1}$. Finally, we generalize the results for convex domains to the domains with only local convexity near $\Gamma_{1}$. 
\end{abstract}

\bigskip
\bigskip

\section{Introduction}  
\label{Sec, introduction}
\subsection{Problem and notations}
In this paper, unless otherwise stated, $\Omega$ represents a bounded open subset in $\m{R}^{n}$ ($n\geq 2$) with $C^{2}$ boundary $\partial\Omega$. $\Gamma_1$ and $\Gamma_2$ denote two disjoint relatively open subsets of $\p\O$. $\p\Gamma_{1}=\p\Gamma_{2}\triangleq \wt{\Gamma}$ is a common $C^{1}$ boundary of $\Gamma_{1}$ and $\Gamma_{2}$. Moreover, $\Gamma_{1}\neq\emptyset$ and $\p\O=\Gamma_{1}\cup\wt{\Gamma}\cup\Gamma_{2}$.  We study the following problem:

\be\label{Prob}
\left\{\begin{array}{lll}
u_{t}(x,t)=\Delta u(x,t) &\text{in}& \Omega\times (0,T], \vspace{0.05in}\\
\frac{\partial u(x,t)}{\partial n(x)}=u^{q}(x,t) &\text{on}& \Gamma_1\times (0,T], \vspace{0.05in}\\
\frac{\partial u(x,t)}{\partial n(x)}=0 &\text{on}& \Gamma_2\times (0,T], \vspace{0.05in}\\
u(x,0)=u_0(x) &\text{in}& \Omega,
\end{array}\right.\ee
where 
\be\label{assumption on prob}
q>1,\, u_0\in C^{1}(\ol{\O}),\, u_0(x)\geq 0,\, u_0(x)\not\equiv 0. \ee
The normal derivative in (\ref{Prob}) is understood in the following way: for any $(x,t)\in\p\O\times(0,T]$,
\be\label{def of normal deri}
\frac{\p u(x,t)}{\p n(x)}\triangleq \lim_{h\rightarrow 0^{+}} (Du)(x_h,t)\cdot\ora{n}(x),\ee
where $Du$ denotes the spatial derivative of $u$, $\ora{n}(x)$ denotes the exterior unit normal vector at $x$ and $x_h\triangleq x-h\ora{n}(x)$ for $x\in\p\O$. Since $\p\O$ is $C^{2}$, $x_h$ belongs to $\O$ when $h$ is positive and sufficiently small. 

Throughout this paper, we write
\be\label{initial max}
M_0= \max_{x\in\ol{\O}}u_0(x)\ee
and denote $M(t)$ to be the supremum of the solution $u$ to (\ref{Prob}) on $\ol{\O}\times[0,t]$:
\be\label{max function at time t}
M(t)=\sup_{(x,\tau)\in \ol{\O}\times[0,t]}u(x,\tau).\ee
$|\Gamma_{1}|$ represents the surface area of $\Gamma_{1}$, that is
\[|\Gamma_{1}|=\int_{\Gamma_{1}}\,dS(x),\]
where $dS(x)$ means the surface integral with respect to the variable $x$. $\Phi$ refers to the fundamental solution to the heat equation:
\be\label{fund soln of heat eq}
\Phi(x,t)=\frac{1}{(4\pi t)^{n/2}}\,\exp\Big(-\frac{|x|^2}{4t}\Big), \quad\forall\, (x,t)\in\m{R}^{n}\times(0,\infty).\ee
In addition, $C=C(a,b\dots)$ and $C_{i}=C_{i}(a,b\dots)$ represent positive constants which only depend on the parameters $a,b\dots$. One should note that $C$ and $C_{i}$ may stand for different constants from line to line. However, $C^{*}=C^{*}(a,b\dots)$ and $C_{i}^{*}=C_{i}^{*}(a,b\dots)$ will represent the constants which are fixed.

When $\Gamma_{1}=\p\O$, the problem (\ref{Prob}) and more general parabolic equations with Neumann boundary conditions have been studied quite a lot. In addition, the Cauchy problems and the Dirichlet boundary value problems  related to the nonlinear blow-up phenomenon of the parabolic type were also investigated. We refer the readers to the surveys \cite{DL00, Lev90} and the books \cite{Fri64, Hu11, QS07}. The topics include the local and global existence and uniqueness of the solutions \cite{Ama86a, Ama86b, ACR-B99, CDW09, LN92, L-GMW91, Wal75, YZ16}; nonexistence of global solutions and the upper bound estimates for the blow-up time \cite{HY94, LN92, Lev73, Lev75, LP74, L-GMW91, PPV-P10, RR97, Wal75, YZ16}; 
lower bound estimates for the blow-up time \cite{LL12, PPV-P10, PS06, PS07, PS09, YZ16, YZ1611};
blow-up sets, blow-up rate and the asymptotic behaviour of the solutions near the blow-up time \cite{FM85, GK85, HY94, LN92, L-GMW91, MW85, RR97, Wei85}.

For the research on the bounds of the blow-up time, the upper bound is usually related to the nonexistence of the global solutions and various methods have been developed. Meanwhile, the lower bound was not studied as much in the past but was paid more attention in recent years. However, the lower bound can be argued to be more useful in practice, since it provides an estimate of the safe time. As an instance, for the problem (\ref{Prob}) which was proposed in \cite{YZ16} to describe the re-entry process to the atmosphere of the Columbia Space Shuttle, the lower bound of the blow-up time would provide a safe time of the landing of the shuttle. In contrast to the upper bound case, not many methods have been explored to deal with the lower bound. In addition, when $\Gamma_{1}$ is a proper subset of $\p\O$, to the authors' knowledge, only two papers \cite{YZ16} and \cite{YZ1611} investigated the relation between the lower bound of the blow-up time and the surface area $|\Gamma_{1}|$. The purpose of this work is to further improve the lower bound estimate of the blow-up time in terms of $|\Gamma_{1}|$, especially when $|\Gamma_{1}|\rightarrow 0$. Before presenting the main results of this paper, let us review what has been known.

The recent paper \cite{YZ16} studied (\ref{Prob}) systematically. According to it (see Theorem 1.3 in \cite{YZ16}), (\ref{Prob}) has a unique classical solution $u$ which is positive. Moreover, if $T^{*}$ denotes the maximal existence time of $u$, then $0<T^{*}<\infty$ and $\lim\limits_{t\nearrow T^{*}}M(t)=\infty$.
In other words, the maximal existence time $T^{*}$ is just the blow-up time of $u$. For simplicity, we will also call $T^{*}$ to be ``the blow-up time''. \cite{YZ16} also provided both upper and lower bounds of $T^{*}$ (see Theorem 1.4 and 1.5 in \cite{YZ16}). For the upper bound, if $\min\limits_{x\in\ol{\O}}u_{0}(x)>0$, then 
\be\label{upper bdd}
T^{*}\leq \frac{1}{(q-1)|\Gamma_1|}\int_{\O}u_0^{1-q}(x)\,dx. \ee
For the lower bound, 
\be\label{lower bdd, small broken part}
T^{*}\geq C^{-\frac{2}{n+2}}\bigg[\ln\Big(|\Gamma_1|^{-1}\Big)-(n+2)(q-1)\ln M_0-\ln(q-1)-\ln C\bigg]^{\frac{2}{n+2}}, \ee
where $C$ is a constant which only depends on $n$, $\O$ and $q$.
In some realistic problems, small $|\Gamma_{1}|$ is of interest. For example in \cite{YZ16}, the motivated model for the study of (\ref{Prob}) is the Columbia space shuttle and $\Gamma_{1}$ stands for the broken part on the left wing of the shuttle during launching, so the surface area $|\Gamma_{1}|$ is expected to be small. As $|\Gamma_{1}|\rightarrow 0^{+}$, the upper bound (\ref{upper bdd}) is of order $|\Gamma_{1}|^{-1}$ while the lower bound (\ref{lower bdd, small broken part}) is only of order $\big[\ln\big(|\Gamma_1|^{-1}\big)\big]^{2/(n+2)}$, so there is a big gap between them. The numerical simulation in \cite{YZ16} is in the same order as the upper bound, so it is desirable to improve the lower bound to at least a polynomial order $|\Gamma_{1}|^{-\alpha}$ for some $\alpha>0$.

In \cite{YZ1611}, by assuming $\O$ is convex, it obtained a lower bound of polynomial order $|\Gamma_{1}|^{-\alpha}$ for any $\alpha<\frac{1}{n-1}$. More precisely, for any $\alpha\in\big[0,\frac{1}{n-1}\big)$, there exists $C=C(n,\O, \alpha)$ such that 
\be\label{lower bdd, convex, general broken part}
T^{*} \geq \frac{C}{(q-1)M_{0}^{q-1}\,|\Gamma_{1}|^{\alpha}}\,\bigg(\min\bigg\{1,\frac{1}{q M_{0}^{q-1}|\Gamma_{1}|^{\alpha}}\bigg\}\bigg)^{\frac{1+(n-1)\alpha}{1-(n-1)\alpha}}. \ee
In addition to the relation between $T^{*}$ and $|\Gamma_{1}|$, 
(\ref{lower bdd, convex, general broken part}) also provides sharp dependence of $T^{*}$ on $q$ and $M_{0}$. As discussed in \cite{YZ1611}, by sending $q\rightarrow 1^{+}$ or $M_{0}\rightarrow 0^{+}$, the order of the lower bound in (\ref{lower bdd, convex, general broken part}) is $(q-1)^{-1}$ or $M_{0}^{-(q-1)}$, both of which are optimal.

Based on the idea in \cite{YZ1611}, this paper will provide a unified method to enhance the lower bound of $T^{*}$ in several aspects (especially the asymptotic behaviour of $T^{*}$ as $|\Gamma_{1}|\rightarrow 0^{+}$) according to the geometric assumptions on $\O$.

\subsection{Main results}
Noticing that the lower bound in (\ref{lower bdd, small broken part}) is negative unless $|\Gamma_{1}|$ or $M_{0}$ is sufficiently small or $q$ is sufficiently close to 1, so it is desirable to derive a lower bound which is always positive. The first result below fulfills this expectation. Moreover, it obtains better asymptotic behavior of the lower bound when $|\Gamma_{1}|\rightarrow 0^{+}$ or $q\rightarrow 1^{+}$.

\begin{theorem}\label{Thm, lower bdd, general domain}
Assume (\ref{assumption on prob}). Let $T^{*}$ be the maximal existence time for (\ref{Prob}). Then there exists a constant $C=C(n,\O)$ such that
\be\label{lower bdd, general domain}
T^{*}\geq \frac{C}{q-1}\ln\Big(1+(2M_{0})^{-4(q-1)}\,|\Gamma_{1}|^{-\frac{2}{n-1}}\Big),\ee
where $M_{0}$ is given by (\ref{initial max}).
\end{theorem}

Let us compare (\ref{lower bdd, general domain}) with (\ref{lower bdd, small broken part}) in more detail on the asymptotic behavior.
\begin{itemize}
\item As $|\Gamma_{1}|\rightarrow 0^{+}$, the order of (\ref{lower bdd, general domain}) is $\ln\big(|\Gamma_{1}|^{-1}\big)$ while the order of (\ref{lower bdd, small broken part}) is only $\big[\ln\big(|\Gamma_{1}|^{-1}\big)\big]^{\frac{2}{n+2}}$.  

\item As $q\rightarrow 1^{+}$, the order of (\ref{lower bdd, general domain}) is $(q-1)^{-1}$, which is optimal since the order of the upper bound (\ref{upper bdd}) is also $(q-1)^{-1}$. However, the order of (\ref{lower bdd, small broken part}) is only $\big(\ln \frac{1}{q-1}\big)^{\frac{2}{n+2}}$. 
\end{itemize}

When the domain $\O$ is convex, for any $\alpha<\frac{1}{n-1}$, \cite{YZ1611} derives the lower bound (\ref{lower bdd, convex, general broken part}) which is of order $|\Gamma_{1}|^{-\alpha}$ as $|\Gamma_{1}|\rightarrow 0^{+}$. The next result in this paper improves the order to be $|\Gamma_{1}|^{-1/(n-1)}$ for $n\geq 3$ and $\big(|\Gamma_{1}|\ln \frac{1}{|\Gamma_{1}|}\big)^{-1}$ for $n=2$ as $|\Gamma_{1}|\rightarrow 0^{+}$.

\begin{theorem}\label{Thm, lower bdd, convex}
Assume (\ref{assumption on prob}). Let $T^{*}$ be the maximal existence time for (\ref{Prob}) and $M_{0}$ be defined as in (\ref{initial max}). Assume $\O$ is convex. Then there exist constants $Y_{0}=Y_{0}(n,\O)$ and $C=C(n,\O)$ such that the following statements hold.
\begin{itemize}
\item Case 1: $n\geq 3$. Denote
\[Y=M_{0}^{q-1}|\Gamma_{1}|^{\frac{1}{n-1}}.\]
If $Y\leq Y_{0}/q$, then 
\be\label{lower bdd, convex, large n}
T^{*}\geq \frac{C}{(q-1)Y}.\ee
 
\item Case 2: $n=2$. Denote
\[Y=M_{0}^{q-1}|\Gamma_{1}|\ln\Big(\frac{1}{|\Gamma_{1}|}+1\Big).\]
If $Y\leq Y_{0}/q$, then 
\be\label{lower bdd, convex, n=2}
T^{*}\geq \frac{C}{(q-1)Y}.\ee
\end{itemize}
\end{theorem}

In some practical situations, the convexity of domain $\O$ is not expected. However, the local convexity near $\Gamma_{1}$ is usually reasonable. Taking the model in \cite{YZ16} as an example again, since $\Gamma_{1}$ is on the left wing of the shuttle, the region near $\Gamma_{1}$ is indeed convex although the whole shuttle is not. Thus it is desirable to generalize Theorem \ref{Thm, lower bdd, convex} to the domains with only local convexity near $\Gamma_{1}$. The third result realizes this goal. Before the statement of the third result, let us explain the meaning of the local convexity near $\Gamma_{1}$.

\begin{definition}[Local convexity near partial boundary]\label{Def, nearby bdry}
Let $\O$ be a bounded open subset in $\m{R}^{n}$ and $\Gamma\subseteq\p\O$. We say $\O$ is locally convex near $\Gamma$ if there exists $d>0$ such that $\text{Conv}\,\big([\Gamma]_{d}\big)\subseteq \ol{\O}$, where 
\be\label{nearby bdry}
[\Gamma]_{d}\triangleq \{x\in\p\O:\text{dist}\,(x,\Gamma)<d\}\ee
denotes the boundary part whose distance to $\Gamma$ is within $d$ and $\text{Conv}([\Gamma]_{d})$ means the convex hull of $[\Gamma]_{d}$.
\end{definition}

Based on this definition, the local convexity near $\Gamma_{1}$ in this paper means $\text{Conv}\,\big([\Gamma_{1}]_{d}\big)\subseteq \ol{\O}$ for some $d>0$.

\begin{theorem}\label{Thm, lower bdd, locally convex}
Assume (\ref{assumption on prob}). Let $T^{*}$ be the maximal existence time for (\ref{Prob}) and $M_{0}$ be defined as in (\ref{initial max}). Assume $\text{Conv}\,\big([\Gamma_{1}]_{d}\big)\subseteq \ol{\O}$ for some $d>0$. Then there exist constants $Y_{0}=Y_{0}(n,\O,d)$ and $C=C(n,\O,d)$ such that the following statements hold.
\begin{itemize}
\item Case 1: $n\geq 3$. Denote
\[Y=M_{0}^{q-1}|\Gamma_{1}|^{\frac{1}{n-1}}.\]
If $Y\leq Y_{0}/q$, then 
\be\label{lower bdd, locally convex, large n}
T^{*}\geq \frac{C}{(q-1)Y|\ln Y|}.\ee
 
\item Case 2: $n=2$. Denote
\[Y=M_{0}^{q-1}|\Gamma_{1}|\ln\Big(\frac{1}{|\Gamma_{1}|}+1\Big).\]
If $Y\leq Y_{0}/q$, then 
\be\label{lower bdd, locally convex, n=2}
T^{*}\geq \frac{C}{(q-1)Y|\ln Y|}.\ee
\end{itemize}
\end{theorem}

To compare Theorem \ref{Thm, lower bdd, locally convex} with Theorem \ref{Thm, lower bdd, convex}, the estimates in Theorem \ref{Thm, lower bdd, locally convex} are almost identical to those in Theorem \ref{Thm, lower bdd, convex} except an extra term $|\ln Y|$ in the denominator. If we look at the proofs, this extra term is due to the lack of the global convexity of $\O$. The outlines of the proofs for Theorem \ref{Thm, lower bdd, convex} and Theorem \ref{Thm, lower bdd, locally convex} are very similar, but the computations in the latter one will be much more complicated due to the lack of the global convexity again.

\subsection{Outline of the approach}
Although this paper deals with domains with three different geometrical assumptions, the methods share many similarities and follow the same outline. Let $M(t)$ be the same as in (\ref{max function at time t}). The basic idea is to chop the range of $M(t)$ into small pieces $[M_{k-1},M_{k}]$ $(k\geq 1)$ and derive a lower bound  $t_{k*}$ for $t_{k}$, the time that $M(t)$ increases from $M_{k-1}$ to $M_{k}$. Suppose such lower bound $t_{k*}$ can be found for $L$ steps ($L$ may be finite or infinite), then $\sum_{k=1}^{L}t_{k*}$ becomes a lower bound for $T^{*}$.  The analysis will be based on the representation formula (\ref{rep formula from any time}). 

The common part of the proofs for Theorems \ref{Thm, lower bdd, general domain}, \ref{Thm, lower bdd, convex} and \ref{Thm, lower bdd, locally convex} is the second paragraph in the proof of Theorem \ref{Thm, lower bdd, general domain} in Section \ref{Sec, proof for general}. After the equation (\ref{abstract recursive ineq, simp}), the proofs will be slightly different due to the geometric properties of the domains. For convenience, we write down the equation (\ref{abstract recursive ineq, simp}) as below. 
\[M_{k} \leq \big[1+4(I_{1}+I_{2})\big]M_{k-1}+4I_{3}M_{k}^{q},\]
where $I_{1}$, $I_{2}$ and $I_{3}$ are defined as in (\ref{def of aux int}). For the estimates on $I_{1}+I_{2}$ and $I_{3}$, we will argue in different ways under the following three cases.
\begin{itemize}
\item[(1)] For a general domain $\O$, Lemma \ref{Lemma, bdry-time int of heat kernel normal deri} implies $I_{1}+I_{2}\leq C\sqrt{t_{k}}$ for some constant $C=C(n,\O)$ and we will use (\ref{bdry-time int bdd}) to bound $I_{3}$.

\item[(2)] For any convex domain $\O$, the identities (\ref{identity, bdry}) and (\ref{identity, convex bdry}) yield $I_{1}+I_{2}=0$ and we will apply Lemma \ref{Lemma, bdry-time int bdd, critical, n large} and Lemma \ref{Lemma, bdry-time int bdd, critical, n=2} to bound $I_{3}$.

\item[(3)] For any domain $\O$ that is locally convex near $\Gamma_{1}$, that is $\text{Conv}\,\big([\Gamma_{1}]_{d}\big)\subseteq \ol{\O}$ for some $d>0$, the identity (\ref{identity, bdry}) and Corollary \ref{Cor, perturbation of id} lead to 
\[I_{1}+I_{2}\leq C\,t_{k}\exp\Big(-\frac{d^{2}}{8t_{k}}\Big)\]
for some constant $C=C(n,\O,d)$. On the other hand, we will exploit Lemma \ref{Lemma, bdry-time int bdd, critical, n large} and Lemma \ref{Lemma, bdry-time int bdd, critical, n=2} again to bound $I_{3}$.
\end{itemize}

Several remarks will be made in sequel.
\begin{itemize}
\item First, since small $t_{k}$ is of interest, the bound for $I_{1}+I_{2}$ in Case (3) is exponential decay as $t_{k}\rightarrow 0$. Due to this fast decay, the result in Case (3) is very close to that in Case (2). In addition, either the result in Case (2) or Case (3) is far better than that in Case (1) where the estimate on $I_{1}+I_{2}$ only decays like $\sqrt{t_{k}}$.

\item Secondly, (\ref{bdry-time int bdd}) implies that 
\[I_{3}\leq \frac{C}{1-(n-1)\alpha} \,|\Gamma_{1}|^{\alpha}\,t_{k}^{\frac{1-(n-1)\alpha}{2}}\]
for some constant $C=C(n,\O)$ and for any $\alpha\in[0,\frac{1}{n-1})$. Lemma \ref{Lemma, bdry-time int bdd, critical, n large} and Lemma \ref{Lemma, bdry-time int bdd, critical, n=2} push the power of $|\Gamma_{1}|$ a little bit further. More precisely,  Lemma \ref{Lemma, bdry-time int bdd, critical, n large} implies 
\[I_{3}\leq C|\Gamma_{1}|^{\frac{1}{n-1}}\]
when $n\geq 3$ and Lemma \ref{Lemma, bdry-time int bdd, critical, n=2} yields
\[I_{3}\leq C|\Gamma_{1}|\ln\Big(\frac{1}{|\Gamma_{1}|}+1\Big)\]
when $n=2$. 

\item Thirdly, for general domains, our method will not gain better lower bound for $T^{*}$ (regarding the order of $|\Gamma_{1}|^{-1}$) by increasing the power of $|\Gamma_{1}|$ in the estimate of $I_{3}$, so we just choose $\alpha=\frac{1}{2(n-1)}$ in (\ref{bdry-time int bdd}) instead of exploiting Lemma \ref{Lemma, bdry-time int bdd, critical, n large} and Lemma \ref{Lemma, bdry-time int bdd, critical, n=2}.

\item Finally, for convex domains or the domains with local convexity near $\Gamma_{1}$, the power of $|\Gamma_{1}|$ in the bound of $I_{3}$ makes a difference in the final lower bound estimate of $T^{*}$ (regarding the order of $|\Gamma_{1}|^{-1}$), so we apply Lemma \ref{Lemma, bdry-time int bdd, critical, n large} and Lemma \ref{Lemma, bdry-time int bdd, critical, n=2} instead of (\ref{bdry-time int bdd}).
\end{itemize}

\subsection{Organization}
The organization of this paper is as follows. Section \ref{Sec, aux lemma} presents some preliminary results which will be used later. Section \ref{Sec, proof for general} verifies Theorem \ref{Thm, lower bdd, general domain} for general domain $\O$. Section \ref{Sec, proof for convex} provides the proof for Theorem \ref{Thm, lower bdd, convex} when the domain $\O$ is convex. Section \ref{Sec, proof for locally convex} justifies Theorem \ref{Thm, lower bdd, locally convex} for the domain $\O$ that is locally convex near $\Gamma_{1}$.

\section{Auxiliary lemmas}
\label{Sec, aux lemma}

\subsection{One identity and its related results}
In \cite{YZ1611}, it mentioned an elementary identity (see Lemma 2.2 in \cite{YZ1611}) about the heat kernel, namely for any $x\in\p\O$ and $t>0$, 
\be\label{identity, bdry}
\int_{\O}\Phi(x-y,t)\,dy-\int_{0}^{t}\int_{\p\O}\frac{\p\Phi(x-y,t-\tau)}{\p n(y)}\,dS(y)\,d\tau=\frac{1}{2}, \quad\forall\, x\in\p\O,\,t>0,\ee
where 
\[\frac{\p\Phi(x-y,t-\tau)}{\p n(y)}\triangleq D_{y}\big[\Phi(x-y,t-\tau)\big]\cdot \ora{n}(y)\]
is the normal derivative. In addition, if $\O$ is convex, then 
\be\label{identity, convex bdry}
\int_{\O}\Phi(x-y,t)\,dy+\int_{0}^{t}\int_{\p\O}\Big| \frac{\p\Phi(x-y,t-\tau)}{\p n(y)}\Big|\,dS(y)\,d\tau=\frac{1}{2}, \quad\forall\, x\in\p\O,\,t>0. \ee

This subsection will derive an intermediate result, Corollary \ref{Cor, perturbation of id}, when the convexity is only assumed near $\Gamma_{1}$ rather than in the whole domain. Before presenting Corollary \ref{Cor, perturbation of id}, we first show an auxiliary lemma. 

\begin{lemma}\label{Lemma, bound for int away from bdry}
Let $\O$ and $\Gamma_1$ be the same as in (\ref{Prob}). Then for any $d>0$, there exists $C=C(n,\O,d)$ such that for any $x\in\ol{\Gamma}_{1}$ and $t>0$, 
\be\label{bound for int away from bdry}
\int_{0}^{t}\int_{\p\O\setminus[\Gamma_{1}]_{d}}\Big|\frac{\p\Phi(x-y,t-\tau)}{\p n(y)}\Big|\,dS(y)\,d\tau\leq C\,t\exp\Big(-\frac{d^{2}}{8t}\Big).\ee
\end{lemma}
\begin{proof} 
In this proof, $C$ denotes a constant which depends only on $n$, $\O$ and $d$. By a change of variable in $\tau$ and the definition of $\Phi$, 
\begin{align}
&\;\;\int_{0}^{t}\int_{\p\O\setminus[\Gamma_{1}]_{d}}\Big|\frac{\p\Phi(x-y,t-\tau)}{\p n(y)}\Big|\,dS(y)\,d\tau \notag\\
= &\;\;\int_{0}^{t}\int_{\p\O\setminus[\Gamma_{1}]_{d}}\Big|\frac{\p\Phi(x-y,\tau)}{\p n(y)}\Big|\,dS(y)\,d\tau \notag\\
\leq &\;\;C \int_{0}^{t}\int_{\p\O\setminus[\Gamma_{1}]_{d}}\frac{|(x-y)\cdot\ora{n}(y)|}{\tau^{\frac{n}{2}+1}}\,\exp\Big(-\frac{|x-y|^{2}}{4\tau}\Big)\,dS(y)\,d\tau. \label{est 1}
\end{align}
Since $\p\O$ is assumed to be $C^{2}$, then $|(x-y)\cdot\ora{n}(y)|\leq C|x-y|^{2}$. In addition, $|x-y|\geq d$ for any $x\in\ol{\Gamma}_{1}$ and $y\in \p\O\setminus[\Gamma_{1}]_{d}$. As a result, 
\begin{eqnarray}
&&\frac{|(x-y)\cdot\ora{n}(y)|}{\tau^{\frac{n}{2}+1}}\,\exp\Big(-\frac{|x-y|^{2}}{4\tau}\Big) \notag\\
&\leq & C|x-y|^{-n}\bigg(\frac{|x-y|^{2}}{\tau}\bigg)^{1+\frac{n}{2}}\exp\Big(-\frac{|x-y|^{2}}{4\tau}\Big) \notag\\
&\leq & C |x-y|^{-n}\exp\Big(-\frac{|x-y|^{2}}{8\tau}\Big) \notag\\
&\leq & C d^{-n}\exp\Big(-\frac{d^{2}}{8\tau}\Big). \label{est 2}
\end{eqnarray}
Plugging (\ref{est 2}) into (\ref{est 1}),   
\begin{align*}
&\;\;\int_{0}^{t}\int_{\p\O\setminus[\Gamma_{1}]_{d}}\Big|\frac{\p\Phi(x-y,t-\tau)}{\p n(y)}\Big|\,dS(y)\,d\tau\\
\leq &\;\; C d^{-n}\int_{0}^{t}\int_{\p\O\setminus[\Gamma_{1}]_{d}}\exp\Big(-\frac{d^{2}}{8\tau}\Big)\,dS(y)\,d\tau\\
\leq &\;\; C d^{-n}|\p\O|\int_{0}^{t}\exp\Big(-\frac{d^{2}}{8\tau}\Big)\,d\tau\\
\leq &\;\; C d^{-n}|\p\O|\, t\exp\Big(-\frac{d^{2}}{8t}\Big).
\end{align*}
\end{proof}
By exploiting Lemma \ref{Lemma, bound for int away from bdry}, the following (\ref{perturbation of id}) is a variant of the identity (\ref{identity, convex bdry}), and it will play the same role in the proof of Theorem \ref{Thm, lower bdd, locally convex} as (\ref{identity, convex bdry}) will do in the proof of Theorem \ref{Thm, lower bdd, convex}.

\begin{corollary}\label{Cor, perturbation of id}
Let $\O$ and $\Gamma_1$ be the same as in (\ref{Prob}). Assume there exists $d>0$ such that $\text{Conv}\,\big([\Gamma_{1}]_{d}\big)\subseteq\ol{\O}$. Then there exists $C=C(n,\O,d)$ such that for any $x\in\ol{\Gamma}_{1}$ and $t>0$,
\be\label{perturbation of id}
\int_{\O}\Phi(x-y,t)\,dy+\int_{0}^{t}\int_{\p\O}\Big|\frac{\p\Phi(x-y,t-\tau)}{\p n(y)}\Big|\,dS(y)\,d\tau\leq \frac{1}{2}+C\,t\exp\Big(-\frac{d^{2}}{8t}\Big).\ee
\end{corollary}
\begin{proof}
Since $x\in\ol{\Gamma}_{1}$ and $\text{Conv}\big([\Gamma_{1}]_{d}\big)\subseteq\ol{\O}$, we have
$$\frac{\p\Phi(x-y,t-\tau)}{\p n(y)}=\frac{C(x-y)\cdot\ora{n}(y)}{(t-\tau)^{n/2+1}}\exp\Big(-\frac{|x-y|^{2}}{4(t-\tau)}\Big)\leq 0,\quad\forall\, y\in [\Gamma_{1}]_{d}.$$
As a result, 
\begin{eqnarray*}
&& \int_{0}^{t}\int_{\p\O}\frac{\p\Phi(x-y,t-\tau)}{\p n(y)}\,dS(y)d\tau+\int_{0}^{t}\int_{\p\O}\Big|\frac{\p\Phi(x-y,t-\tau)}{\p n(y)}\Big|\,dS(y)d\tau \\
&=& \int_{0}^{t}\int_{\p\O\setminus[\Gamma_{1}]_{d}}\frac{\p\Phi(x-y,t-\tau)}{\p n(y)}+\Big|\frac{\p\Phi(x-y,t-\tau)}{\p n(y)}\Big|\,dS(y)\,d\tau\\
&\leq & C\,t\exp\Big(-\frac{d^{2}}{8t}\Big),
\end{eqnarray*}
where the last inequality is due to Lemma \ref{Lemma, bound for int away from bdry}. Therefore 
\begin{eqnarray*}
&& \int_{\O}\Phi(x-y,t)\,dy+\int_{0}^{t}\int_{\p\O}\Big|\frac{\p\Phi(x-y,t-\tau)}{\p n(y)}\Big|\,dS(y)\,d\tau\\
&\leq & \int_{\O}\Phi(x-y,t)dy-\int_{0}^{t}\int_{\p\O}\frac{\p\Phi(x-y,t-\tau)}{\p n(y)}dS(y)d\tau+C\,t\exp\Big(-\frac{d^{2}}{8t}\Big)\\
&=& \frac{1}{2}+C\,t\exp\Big(-\frac{d^{2}}{8t}\Big),
\end{eqnarray*}
where the last equality is because of (\ref{identity, bdry}).
\end{proof}

\subsection{Estimate for the boundary-time integral of the heat kernel}
The estimate for the boundary-time integral of the heat kernel is a basic tool in the derivation of the lower bound in (\ref{lower bdd, convex, general broken part}). More precisely (see Lemma 2.3 in \cite{YZ1611}), there exists $C=C(n,\O)$ such that for any $\Gamma\subseteq\p\O$, $\alpha\in\big[0,\frac{1}{n-1}\big)$, $x\in\p\O$ and $t>0$,
\be\label{bdry-time int bdd}
\int_{0}^{t}\int_{\Gamma}\Phi(x-y,t-\tau)\,dS(y)\,d\tau\leq \frac{C}{1-(n-1)\alpha} \,|\Gamma|^{\alpha}\,t^{\frac{1-(n-1)\alpha}{2}}.\ee
According to the method in \cite{YZ1611}, the power $\alpha$ in (\ref{bdry-time int bdd}) determines the power on $|\Gamma_{1}|^{-1}$ of the lower bound for $T^{*}$ in (\ref{lower bdd, convex, general broken part}). However, the range of the power $\alpha$ in (\ref{bdry-time int bdd}) missed $\frac{1}{n-1}$ since the coefficient will blow up as $\alpha\nearrow \frac{1}{n-1}$. So it is natural to ask whether $\alpha$ can be taken as $\frac{1}{n-1}$ by other methods. In this subsection, the above expectation will be justified for $n\geq 3$ in Lemma \ref{Lemma, bdry-time int bdd, critical, n large} and for $n=2$ (with an extra log term and bounded time $t$) in Lemma \ref{Lemma, bdry-time int bdd, critical, n=2}.

We first introduce a simple fact which can be regarded as a rearrangement result.
\begin{lemma}\label{Lemma, mass ineq}
Let $n\geq 1$ and $f:(0,\infty)\rightarrow [0,\infty)$ be a decreasing function. Then for any bounded subset $U$ of $\m{R}^{n}$ and for any $x\in\m{R}^{n}$, 
\be\label{mass ineq}
\int_{U}f(|x-y|)\,dy\leq \int_{B_{R}(0)}f(|z|)\,dz \ee 
where $R$ satisfies $|B_{R}(0)|=|U|$ (namely the volume of $B_{R}(0)$ equals the volume of $U$).
\end{lemma} 
\begin{proof}
Define \[U_{1}=U-\{x\}.\] 
Then by a change of variable $z=y-x$,
\begin{align}
\int_{U}f(|x-y|)\,dy &= \int_{U_{1}}f(|z|)\,dz \notag\\
&= \int_{U_{1}\cap B_{R}(0)}f(|z|)\,dz+\int_{U_{1}\backslash B_{R}(0)}f(|z|)\,dz \notag\\
&\triangleq J_{1}+J_{2},  \label{split from inner and outer}
\end{align}
Since $f$ is decreasing, 
\begin{align*}
J_{2}\leq f(R)|U_{1}\backslash B_{R}(0)|.
\end{align*}
Due to the definition of $R$, $|B_{R}(0)|=|U|=|U_{1}|$. So we have $|B_{R}(0)\backslash U_{1}|=|U_{1}\backslash B_{R}(0)|$. As a result, 
\be\label{est for outside int} J_{2}\leq f(R)|B_{R}(0)\backslash U_{1}|\leq \int_{B_{R}(0)\backslash U_{1}}f(|z|)\,dz,\ee
where the last inequality is again due to the decay of $f$. 
Combining (\ref{split from inner and outer}) and (\ref{est for outside int}), we finish the proof.
\end{proof}

\begin{definition}\label{Def, bdry given by graph}
Let $\O$ be a bounded, open subset of $\m{R}^{n}$ with $C^{1}$ boundary. Let $\Gamma$ be a subset of $\p\O$. We say $\Gamma$ is given by a graph if (upon relabelling and reorienting the coordinates axes) there exists a bounded subset $U\subseteq\m{R}^{n-1}$ and a $C^{1}$ function $\phi:\m{R}^{n-1}\rightarrow \m{R}$ such that 
$$\Gamma=\{(\t{y}, \phi(\t{y})):\t{y}\in U\}.$$
\end{definition}

In the following, for any $x\in\m{R}^{n}$, we will decompose it to be $x=(\t{x},x_{n})$, where $\t{x}$ denotes the first $n-1$ components of $x$. 

\begin{lemma}\label{Lemma, bdry int higher dim}
Let $\O$ be a bounded, open subset of $\m{R}^{n}(n\geq 3)$ with $C^{1}$ boundary. Let $\Gamma$ be a subset of $\p\O$ that is given by a graph as in Definition \ref{Def, bdry given by graph}. Then there exists a constant $C=C(n,||\nabla \phi||_{L^{\infty}(U)})$, where $\phi$ and $U$ are the same as in Definition \ref{Def, bdry given by graph}, such that for any $x\in\m{R}^{n}$,
\[\int_{\Gamma}\frac{1}{|x-y|^{n-2}}\,dS(y)\leq C|\Gamma|^{1/(n-1)}.\]
\end{lemma}
\begin{proof}
By Definition \ref{Def, bdry given by graph}, without loss of generality, we can assume there exists a $C^{1}$ function $\phi:\m{R}^{n-1}\rightarrow \m{R}$ and a bounded subset $U$ of $\m{R}^{n-1}$ such that 
\be\label{para gamma, high dim}\Gamma=\{(\t{y}, \phi(\t{y})):\t{y}\in U\}.\ee
Thus, 
\begin{align*}
\int_{\Gamma}\frac{1}{|x-y|^{n-2}}\,dS(y) &=\int_{U}\frac{\sqrt{1+|\nabla \phi(\t{y})|^{2}}}{|(\t{x}, x_{n})-(\t{y},\phi(\t{y}))|^{n-2}}\,d\t{y}\\
&\leq \int_{U}\frac{\sqrt{1+|\nabla \phi(\t{y})|^{2}}}{|\t{x}-\t{y}|^{n-2}}\,d\t{y}\\
&\leq C\int_{U}\frac{1}{|\t{x}-\t{y}|^{n-2}}\,d\t{y}.
\end{align*}
Define 
\[f(r)=\frac{1}{r^{n-2}},\quad\forall\, r>0.\]
Then it follows from Lemma \ref{Lemma, mass ineq} that 
\begin{align*}
\int_{U}\frac{1}{|\t{x}-\t{y}|^{n-2}}\,d\t{y} &= \int_{U}f(|\t{x}-\t{y}|)\,d\t{y}\\
&\leq \int_{B_{R}(0)}f(|\t{z}|)\,d\t{z}\\
&=CR= C|U|^{1/(n-1)}.
\end{align*}
Again by the parametrization (\ref{para gamma, high dim}), it is readily seen that $|U|\leq |\Gamma|$. Hence, 
\[\int_{\Gamma}\frac{1}{|x-y|^{n-2}}\,dS(y)\leq C|U|^{1/(n-1)}\leq C|\Gamma|^{1/(n-1)}.\]
\end{proof}

\begin{corollary}\label{Cor, bound for int on gamma, high dim}
Let $\O$ be a bounded open subset of $\m{R}^{n}(n\geq 3)$ with $C^{1}$ boundary. Let $\Gamma$ be any subset of $\p\O$. Then there exists a constant $C=C(n,\O)$ such that for any $x\in\m{R}^{n}$,
\[\int_{\Gamma}\frac{1}{|x-y|^{n-2}}\,dS(y)\leq C|\Gamma|^{1/(n-1)}.\]
\end{corollary}
\begin{proof}
Since $\p\O$ is $C^{1}$, for any point $x_{0}\in\p\O$, the boundary part of $\O$ near $x_{0}$ is given by a graph as in Definition \ref{Def, bdry given by graph}). Therefore we can split $\p\O$ into finite pieces: 
\be\label{decomp of bdry}\p\O=\bigcup_{i=1}^{K}A_{i},\ee
where each $A_{i}(1\leq i\leq K)$ is given by the graph of some $C^{1}$ function $\phi_{i}$ on some bounded set $U_{i}\subseteq\m{R}^{n-1}$. The number of total pieces $K$ and $||\nabla \phi_{i}||_{L^{\infty}(U_{i})}$ only depend on $\O$. 

For any $1\leq i\leq K$, $\Gamma\cap A_{i}$ is also a boundary part given by a graph. Therefore by Lemma \ref{Lemma, bdry int higher dim}, there exists a constant $C=C(n,\O)$ such that for any $1\leq i\leq K$,
\[\int_{\Gamma\cap A_{i}}\frac{1}{|x-y|^{n-2}}dS(y)\leq C|\Gamma\cap A_{i}|^{1/(n-1)}.\]
Hence,
\begin{align*}
\int_{\Gamma}\frac{1}{|x-y|^{n-2}}dS(y) &\leq \sum_{i=1}^{K}\int_{\Gamma\cap A_{i}}\frac{1}{|x-y|^{n-2}}dS(y)\\
&\leq C\sum_{i=1}^{K}|\Gamma\cap A_{i}|^{1/(n-1)}\\
&\leq CK|\Gamma|^{1/(n-1)}=C|\Gamma|^{1/(n-1)}.
\end{align*}
\end{proof}

Lemma \ref{Lemma, bdry int higher dim} and Corollary \ref{Cor, bound for int on gamma, high dim} will be applied to show the desired Lemma \ref{Lemma, bdry-time int bdd, critical, n large} which pushes the power $\alpha$ in (\ref{bdry-time int bdd}) to $\frac{1}{n-1}$ when $n\geq 3$.

\begin{lemma}\label{Lemma, bdry-time int bdd, critical, n large}
Let $\O$ be a bounded open subset of $\m{R}^{n}(n\geq 3)$ with $C^{1}$ boundary. Let $\Gamma$ be any subset of $\p\O$.  Then there exists $C=C(n,\O)$ such that for any $x\in\m{R}^{n}$ and $t\geq 0$, \be\label{bdry-time int bdd, critical, n large}
\int_{0}^{t}\int_{\Gamma}\Phi(x-y,t-\tau)\,dS(y)\,d\tau\leq C|\Gamma|^{1/(n-1)}.\ee
\end{lemma}
\begin{proof}
In this proof, unless otherwise stated, $C$ represents constants which only depend on $n$ and $\O$. First, by the explicit formula (\ref{fund soln of heat eq}) of $\Phi$ and a change of variable in $\tau$, we have 
\[\int_{0}^{t}\int_{\Gamma}\Phi(x-y,t-\tau)\,dS(y)\,d\tau=C\int_{\Gamma}\int_{0}^{t}\tau^{-n/2}e^{-|x-y|^{2}/(4\tau)}\,d\tau\,dS(y).\]
Then by the change of variable $s=|x-y|^{2}/(4\tau)$ for $\tau$, 
\be\label{simple form of int}\begin{split}
&\int_{\Gamma}\int_{0}^{t}\tau^{-n/2}e^{-|x-y|^{2}/(4\tau)}\,d\tau\,dS(y) \\
\leq\;\; & C\int_{\Gamma}\frac{1}{|x-y|^{n-2}}\int_{|x-y|^{2}/(4t)}^{\infty}s^{\frac{n}{2}-2}e^{-s}\,ds\,dS(y).
\end{split}\ee
Since $n\geq 3$, $s^{\frac{n}{2}-2}e^{-s}$ is integrable on $(0,\infty)$. As a result, 
\begin{eqnarray*}
&&\int_{\Gamma}\frac{1}{|x-y|^{n-2}}\int_{|x-y|^{2}/(4t)}^{\infty}s^{\frac{n}{2}-2}e^{-s}\,ds\,dS(y) \\
&\leq &\int_{\Gamma}\frac{1}{|x-y|^{n-2}}\int_{0}^{\infty}s^{\frac{n}{2}-2}e^{-s}\,ds\,dS(y)\\
&=& C\int_{\Gamma}\frac{1}{|x-y|^{n-2}}\,dS(y).
\end{eqnarray*}
Now applying Corollary \ref{Cor, bound for int on gamma, high dim}, 
\[\int_{\Gamma}\frac{1}{|x-y|^{n-2}}\,dS(y)\leq C|\Gamma|^{1/(n-1)}.\]
\end{proof}

The following Lemma \ref{Lemma, bdry int 2 dim}, Corollary \ref{Cor, bound for int on gamma, dim 2} and Lemma \ref{Lemma, bdry-time int bdd, critical, n=2} are parallel results as Lemma \ref{Lemma, bdry int higher dim}, Corollary \ref{Cor, bound for int on gamma, high dim} and Lemma \ref{Lemma, bdry-time int bdd, critical, n large}, but they deal with dimension $n=2$ rather than $n\geq 3$.

\begin{lemma}\label{Lemma, bdry int 2 dim}
Let $\O$ be a bounded, open subset of $\m{R}^{2}$ with $C^{1}$ boundary. Let $\Gamma$ be any subset of $\p\O$ that is given by a graph as in Definition \ref{Def, bdry given by graph}. Then there exists a constant $C=C(\O, ||\nabla \phi||_{L^{\infty}(U)})$, where $\phi$ and $U$ are the same as those in Definition \ref{Def, bdry given by graph}, such that for any $x\in\ol{\O}$,
\[\int_{\Gamma}\ln\Big(\frac{d_{\O}}{|x-y|}\Big)\,dS(y)\leq C|\Gamma|\ln\Big(\frac{1}{|\Gamma|}+1\Big),\]
where $d_{\O}$ denotes the diameter of $\O$.
\end{lemma}

\begin{proof}
By Definition \ref{Def, bdry given by graph}, without loss of generality, we can assume there exists a $C^{1}$ function $\phi:\m{R}\rightarrow\m{R}$ and a bounded set $U\subseteq\m{R}$ such that 
\be\label{para gamma, dim 2}\Gamma=\{(\t{y}, \phi(\t{y})):\t{y}\in U\}.\ee
In addition, we define 
\be\label{def of f, dim 2}
f(r)=\left\{\begin{array}{ll}
\ln\big(\frac{d_{\O}}{r}\big), & 0<r\leq d_{\O},\vspace{0.05in}\\
0, & r>d_{\O}. 
\end{array}\right.\ee

Since $x=(\t{x},x_{n})\in\ol{\O}$, then for any $(\t{y}, \phi(\t{y}))\in\Gamma$,
\[|\t{x}-\t{y}|\leq |(\t{x},x_{n})-(\t{y}, \phi(\t{y}))|\leq d_{\O}.\]
As a result, 
\begin{align}\label{reduce to f}
\int_{\Gamma}\ln\Big(\frac{d_{\O}}{|x-y|}\Big)\,dS(y) &= \int_{U}\ln\Big(\frac{d_{\O}}{|(\t{x},x_{n)}-(\t{y},\phi(\t{y}))|}\Big)\sqrt{1+|\nabla \phi(\t{y})|^{2}}\,d\t{y} \notag\\
&\leq C\int_{U}\ln\Big(\frac{d_{\O}}{|\t{x}-\t{y}|}\Big)\,d\t{y} \notag\\
&=C\int_{U}f(|\t{x}-\t{y}|)\,d\t{y}.
\end{align}
Now it follows from Lemma \ref{Lemma, mass ineq} that 
\begin{align}\label{reduce to polar form}
\int_{U}f(|\t{x}-\t{y}|)\,d\t{y} &\leq \int_{B_{R}(0)}f(|\t{z}|)\,d\t{z} \notag\\
&=2\int_{0}^{R}f(r)\,dr,
\end{align}
where $|B_{R}(0))|=|U|$, namely $2R=|U|$. For any $\t{y}_{1}, \t{y}_{2}\in U$, we have
\[|\t{y}_{1}-\t{y}_{2}|\leq |(\t{y}_{1},\phi(\t{y}_{1}))-(\t{y}_{2},\phi(\t{y}_{2}))|\leq d_{\O},\]
which implies $\text{diam}(U)\leq d_{\O}$. Moreover, since $U\subseteq\m{R}$, then $|U|\leq \text{diam}(U)$. Thus, $R=|U|/2\leq d_{\O}/2$. So it follows from (\ref{def of f, dim 2}) that 
\begin{align}\label{int of f by R}
\int_{0}^{R}f(r)\,dr &= \int_{0}^{R}\ln\Big(\frac{d_{\O}}{r}\Big)\,dr \notag\\
&=R \Big[\ln\Big(\frac{d_{\O}}{R}\Big)+1\Big].
\end{align}

Again by the parametrization (\ref{para gamma, dim 2}), it is readily seen that $|U|\leq |\Gamma|$. Therefore, $$R\leq \min\Big\{\frac{|\Gamma|}{2}, \frac{d_{\O}}{2}\Big\}.$$
Define 
\[g(r)=r\Big[\ln\Big(\frac{d_{\O}}{r}\Big)+1\Big], \quad\forall\, r>0.\]
Then $g$ is increasing when $r\in(0,d_{\O}]$ and (\ref{int of f by R}) implies $\int_{0}^{R}f(r)\,dr=g(R)$. Next, we will estimate $g(R)$ in the following two situations.
\begin{itemize}
\item $|\Gamma|\leq d_{\O}$. 
\begin{align}
g(R) \leq g(|\Gamma|)&=|\Gamma|\Big[\ln\Big(\frac{d_{\O}}{|\Gamma|}\Big)+1\Big] \notag\\
&\leq C|\Gamma|\ln\Big(\frac{1}{|\Gamma|}+1\Big) \label{est for small Gamma}
\end{align}
for some constant $C$ only depending on $\O$.

\item $|\Gamma|> d_{\O}$. 
\begin{align*}
g(R) &\leq g(d_{\O})=d_{\O}.
\end{align*}
Define 
\be\label{def of h}
h(r)=r\ln\Big(\frac{1}{r}+1\Big), \quad\forall\, r>0.\ee
Then \[h''(r)=-\frac{1}{r(1+r)^2}<0, \quad\forall\, r>0.\]
This implies $h'(r)>0$ for any $r>0$, since $\lim\limits_{r\rightarrow\infty}h'(r)=0$. Hence, $h$ is an increasing function and 
\[|\Gamma|\ln\Big(\frac{1}{|\Gamma|}+1\Big) =h(|\Gamma|)\geq h(d_{\O})=d_{\O}\ln\Big(\frac{1}{d_{\O}}+1\Big).\]
Thus,  
\be\label{est for large Gamma}g(R)\leq C|\Gamma|\ln\Big(\frac{1}{|\Gamma|}+1\Big),\ee
where $C=1/\ln\big(\frac{1}{d_{\O}}+1\big)$ is a constant depending only on $\O$.
\end{itemize}

Combining (\ref{reduce to f}), (\ref{reduce to polar form}), (\ref{est for small Gamma}) and (\ref{est for large Gamma}), the conclusion follows.
\end{proof}

\begin{corollary}\label{Cor, bound for int on gamma, dim 2}
Let $\O$ be a bounded, open subset of $\m{R}^{2}$ with $C^{1}$ boundary. Let $\Gamma$ be any subset of $\p\O$. Then there exists a constant $C=C(\O)$ such that for any $x\in\ol{\O}$,
\[\int_{\Gamma}\ln\Big(\frac{d_{\O}}{|x-y|}\Big)\,dS(y)\leq C|\Gamma|\ln\Big(\frac{1}{|\Gamma|}+1\Big),\]
where $d_{\O}$ denotes the diameter of $\O$.
\end{corollary}
\begin{proof}
Similar to the proof of Corollary \ref{Cor, bound for int on gamma, high dim}, we first decompose $\p\O$ as that in (\ref{decomp of bdry}). Then 
\begin{align*}
\int_{\Gamma}\ln\Big(\frac{d_{\O}}{|x-y|}\Big)\,dS(y) &\leq \sum_{i=1}^{K}\int_{\Gamma\cap A_{i}}\ln\Big(\frac{d_{\O}}{|x-y|}\Big)\,dS(y).
\end{align*}
Since each $\Gamma\cap A_{i}$ is given by a graph, we can apply Lemma \ref{Lemma, bdry int 2 dim} to conclude there exists a constant $C=C(\O)$ such that for each $1\leq i\leq K$,
\[\int_{\Gamma\cap A_{i}}\ln\Big(\frac{d_{\O}}{|x-y|}\Big)\,dS(y)\leq C|\Gamma\cap A_{i}|\ln\Big(\frac{1}{|\Gamma\cap A_{i}|}+1\Big).\]
Recalling the function $h$ defined in (\ref{def of h}) is an increasing function, so
\[|\Gamma\cap A_{i}|\ln\Big(\frac{1}{|\Gamma\cap A_{i}|}+1\Big)\leq |\Gamma|\ln\Big(\frac{1}{|\Gamma|}+1\Big).\]
As a result,
\[\int_{\Gamma}\ln\Big(\frac{d_{\O}}{|x-y|}\Big)\,dS(y)\leq C|\Gamma|\ln\Big(\frac{1}{|\Gamma|}+1\Big).\]
\end{proof}

Next, Lemma \ref{Lemma, bdry int 2 dim} and Corollary \ref{Cor, bound for int on gamma, dim 2} will be applied to show our desired Lemma \ref{Lemma, bdry-time int bdd, critical, n=2} which is an improvement of (\ref{bdry-time int bdd}) when $n=2$.

\begin{lemma}\label{Lemma, bdry-time int bdd, critical, n=2}
Let $\O$ be a bounded, open subset of $\m{R}^{2}$ with $C^{1}$ boundary. Let $\Gamma$ be any subset of $\p\O$. Then there exists $C=C(\O)$ such that for any $x\in\ol{\O}$ and $t\in[0,1]$, 
\be\label{bdry-time int bdd, critical, n=2}
\int_{0}^{t}\int_{\Gamma}\Phi(x-y,t-\tau)\,dS(y)\,d\tau\leq C|\Gamma|\ln\Big(\frac{1}{|\Gamma|}+1\Big).\ee
\end{lemma}
\begin{proof}
We proceed similarly as that in the proof of Lemma \ref{Lemma, bdry-time int bdd, critical, n large} until (\ref{simple form of int}). Next, the situation is different since $s^{n/2-2}e^{-s}$ is not integrable near $s=0$ when $n=2$. For convenience, we rewrite (\ref{simple form of int}) when $n=2$ as following:
\be\label{simple form of int, n=2}
\int_{\Gamma}\int_{0}^{t}\tau^{-1}e^{-|x-y|^{2}/(4\tau)}\,d\tau\,dS(y)\leq C\int_{\Gamma}\int_{|x-y|^{2}/(4t)}^{\infty}s^{-1}e^{-s}\,ds\,dS(y).\ee
Since $t\leq 1$ and $x\in\ol{\O}$, $|x-y|^{2}/(4t)\geq |x-y|^{2}/4$. Thus,
\begin{align*}
\int_{|x-y|^{2}/(4t)}^{\infty}s^{-1}e^{-s}\,ds &\leq \int_{|x-y|^{2}/4}^{\infty}s^{-1}e^{-s}\,ds\\
&=\int_{|x-y|^{2}/4}^{d_{\O}^{2}}s^{-1}e^{-s}\,ds+\int_{d_{\O}^{2}}^{\infty}s^{-1}e^{-s}\,ds\\
&\leq \int_{|x-y|^{2}/4}^{d_{\O}^{2}}s^{-1}\,ds+ \frac{1}{d_{\O}^{2}}\int_{d_{\O}^{2}}^{\infty}e^{-s}\,ds\\
&=2\ln\Big(\frac{d_{\O}}{|x-y|}\Big)+C.
\end{align*}
As a result,
\be\label{two terms, n=2}\int_{\Gamma}\int_{|x-y|^{2}/(4t)}^{\infty}s^{-1}e^{-s}\,ds\,dS(y)\leq 2\int_{\Gamma}\ln\Big(\frac{d_{\O}}{|x-y|}\Big)\,dS(y)+C|\Gamma|.\ee
Now applying Corollary \ref{Cor, bound for int on gamma, dim 2}, 
\[\int_{\Gamma}\ln\Big(\frac{d_{\O}}{|x-y|}\Big)\,dS(y)\leq C|\Gamma|\ln\Big(\frac{1}{|\Gamma|}+1\Big).\]
Finally noticing that 
\begin{align*}
|\Gamma| &\leq \frac{1}{\ln\Big(\frac{1}{|\p\O|}+1\Big)}|\Gamma|\ln\Big(\frac{1}{|\Gamma|}+1\Big)\\
&=C|\Gamma|\ln\Big(\frac{1}{|\Gamma|}+1\Big),
\end{align*}
the lemma is proved.
\end{proof}

\section{Proof of Theorem \ref{Thm, lower bdd, general domain}}
\label{Sec, proof for general}
The starting point of the proofs in this paper is the representation formula of the solution $u$ (see Lemma A.1 in \cite{YZ1611}): for any $T\in [0,T^{*})$ and $(x,t)\in\p\O\times[0,T^{*}-T)$,
\begin{eqnarray}
u(x,T+t) &=& 2\int_{\O}\Phi(x-y,t)\,u(y,T)\,dy \notag\\
&& -2\int_{0}^{t}\int_{\p\O}\frac{\p\Phi(x-y,t-\tau)}{\p n(y)}\,u(y,T+\tau)\,dS(y)\,d\tau \notag\\
&& +2\int_{0}^{t}\int_{\Gamma_1}\Phi(x-y,t-\tau)\,u^{q}(y,T+\tau)\,dS(y)\,d\tau.  \label{rep formula from any time}
\end{eqnarray}
To estimate the integral of $\frac{\p\Phi(x-y,t-\tau)}{\p n(y)}$ on $\p\O\times[0,t]$, we apply the lemma below.

\begin{lemma}\label{Lemma, bdry-time int of heat kernel normal deri}
There exists $C=C(n,\O)$ such that for any $x\in\p\O$ and $t>0$,
\be\label{bdry-time int of heat kernel normal deri}
\int_{0}^{t}\int_{\p\O}\bigg|\frac{\p\Phi(x-y,t-\tau)}{\p n(y)}\bigg|\,dS(y)\,d\tau\leq C\sqrt{t}.\ee
\end{lemma}
\begin{proof}
By the definition of $\Phi$, 
\[\bigg|\frac{\p\Phi(x-y,t-\tau)}{\p n(y)}\bigg|=\frac{C|(x-y)\cdot \ora{n}(y)|}{(t-\tau)^{\frac{n}{2}+1}}\exp\Big(-\frac{|x-y|^{2}}{4(t-\tau)}\Big).\]
Since $\p\O$ is assumed to be $C^{2}$, there exists a constant $C$ such that $|(x-y)\cdot \ora{n}(y)|\leq C|x-y|^{2}$ for any $x,y\in \p\O$. As a result, 
\[\bigg|\frac{\p\Phi(x-y,t-\tau)}{\p n(y)}\bigg|\leq \frac{C|x-y|^{2}}{(t-\tau)^{\frac{n}{2}+1}}\exp\Big(-\frac{|x-y|^{2}}{4(t-\tau)}\Big).\]
Noticing the term
\[\frac{|x-y|^{2}}{t-\tau}\exp\Big(-\frac{|x-y|^{2}}{8(t-\tau)}\Big)\]
is bounded by some constant, so 
\[\bigg|\frac{\p\Phi(x-y,t-\tau)}{\p n(y)}\bigg|\leq \frac{C}{(t-\tau)^{n/2}}\exp\Big(-\frac{|x-y|^{2}}{8(t-\tau)}\Big).\]
Thus,
\begin{eqnarray*}
&&\int_{0}^{t}\int_{\p\O}\bigg|\frac{\p\Phi(x-y,t-\tau)}{\p n(y)}\bigg|\,dS(y)\,d\tau \\
&\leq & C\int_{0}^{t}\int_{\p\O}\frac{1}{(t-\tau)^{n/2}}\exp\Big(-\frac{|x-y|^{2}}{8(t-\tau)}\Big)\,dS(y)\,d\tau\\
&=& C\int_{0}^{t}\int_{\p\O}\frac{1}{\tau^{n/2}}\exp\Big(-\frac{|x-y|^{2}}{8\tau}\Big)\,dS(y)\,d\tau.
\end{eqnarray*}

By the change of variable $\sigma=2\tau$,
\begin{eqnarray*}
&&\int_{0}^{t}\int_{\p\O}\frac{1}{\tau^{n/2}}\exp\Big(-\frac{|x-y|^{2}}{8\tau}\Big)\,dS(y)\,d\tau \\
&=& C \int_{0}^{2t}\int_{\p\O}\frac{1}{\sigma^{n/2}}\exp\Big(-\frac{|x-y|^{2}}{4\sigma}\Big)\,dS(y)\,d\sigma\\
&=& C \int_{0}^{2t}\int_{\p\O}\Phi(x-y,\sigma)\,dS(y)\,d\sigma\\
&=& C \int_{0}^{2t}\int_{\p\O}\Phi(x-y,2t-\sigma)\,dS(y)\,d\sigma.
\end{eqnarray*}
Finally, invoking (\ref{bdry-time int bdd}) with $\Gamma=\p\O$ and $\alpha=0$, the proof is finished.
\end{proof}

\begin{proof}[{\bf Proof of Theorem \ref{Thm, lower bdd, general domain}}]
In this proof, $C$ will denote constants which only depend on $n$ and $\O$, the values of $C$ may be different in different places. But $C^{*}$ and $C_{i}^{*}(i\geq 1)$ will represent fixed constants which only depend on $n$ and $\O$. $M(t)$ represents the same function as in (\ref{max function at time t}). 

For any strictly increasing sequence $\{M_{k}\}_{k\geq 0}$ whose initial term is the same as the $M_{0}$ defined in (\ref{initial max}), we denote $T_{k}$ to be the first time that $M(t)$ reaches $M_{k}$. Obviously, $T_{0}=0$. For any $k\geq 1$, define
\be\label{difference in time, kth step, general}
t_{k}=T_{k}-T_{k-1} \ee
to be the time spent in the kth step. By the maximum principle and the Hopf lemma, there exists $x^{k}\in\ol{\Gamma}_{1}$ such that 
\be\label{maximum pt, kth, general}
u(x^{k},T_{k})=M_{k}.\ee 
Applying the representation formula (\ref{rep formula from any time}) with $T=T_{k-1}$ and $(x,t)=(x^{k},t_{k})$, then
\begin{align}
u(x^{k},T_{k}) = &\,\, 2\int_{\O}\Phi(x^{k}-y,t_{k})\,u(y,T_{k-1})\,dy \notag\\
&-2\int_{0}^{t_{k}}\int_{\p\O}\frac{\p\Phi(x^{k}-y,t_{k}-\tau)}{\p n(y)}\,u(y,T_{k-1}+\tau)\,dS(y)\,d\tau \notag\\
& +2\int_{0}^{t_{k}}\int_{\Gamma_1}\Phi(x^{k}-y,t_{k}-\tau)\,u^{q}(y,T_{k-1}+\tau)\,dS(y)\,d\tau. \label{rep formula for kth step, general}
\end{align}
Combining (\ref{maximum pt, kth, general}) and (\ref{rep formula for kth step, general}),
\begin{eqnarray*}
M_{k} &\leq & 2M_{k-1}\int_{\O}\Phi(x^{k}-y,t_{k})\,dy\\
&&+2M_{k}\int_{0}^{t_{k}}\int_{\p\O}\Big|\frac{\p\Phi(x^{k}-y,t_{k}-\tau)}{\p n(y)}\Big|\,dS(y)\,d\tau \\
&& +2M_{k}^{q}\int_{0}^{t_{k}}\int_{\Gamma_1}\Phi(x^{k}-y,t_{k}-\tau)\,dS(y)\,d\tau.
\end{eqnarray*}
Replacing the term $\int_{\O}\Phi(x^{k}-y,t_{k})\,dy$ by the identity (\ref{identity, bdry}), then
\begin{align*}
M_{k} \leq &\,\, 2M_{k-1}\bigg[\frac{1}{2}+\int_{0}^{t_{k}}\int_{\p\O}\frac{\p\Phi(x^{k}-y,t_{k}-\tau)}{\p n(y)}\,dS(y)\,d\tau\bigg]\\
& +2M_{k}\int_{0}^{t_{k}}\int_{\p\O}\bigg|\frac{\p\Phi(x^{k}-y,t_{k}-\tau)}{\p n(y)}\bigg|\,dS(y)\,d\tau \\
& +2M_{k}^{q}\int_{0}^{t_{k}}\int_{\Gamma_1}\Phi(x^{k}-y,t_{k}-\tau)\,dS(y)\,d\tau.
\end{align*}
Moving the term on the second line of the right hand side to the left, we obtain
\be\label{abstract recursive ineq}
(1-2I_{1})M_{k}\leq (1+2I_{2})M_{k-1}+2I_{3}M_{k}^{q},\ee
where
\begin{eqnarray}
I_{1} &=& \int_{0}^{t_{k}}\int_{\p\O}\bigg|\frac{\p\Phi(x^{k}-y,t_{k}-\tau)}{\p n(y)}\bigg|\,dS(y)\,d\tau, \notag\\
I_{2} &=& \int_{0}^{t_{k}}\int_{\p\O}\frac{\p\Phi(x^{k}-y,t_{k}-\tau)}{\p n(y)}\,dS(y)\,d\tau, \label{def of aux int}\\
I_{3} &=& \int_{0}^{t_{k}}\int_{\Gamma_1}\Phi(x^{k}-y,t_{k}-\tau)\,dS(y)\,d\tau. \notag
\end{eqnarray}
It is readily seen that $|I_{2}|\leq I_{1}$. In addition, by Lemma \ref{Lemma, bdry-time int of heat kernel normal deri}, there exists a constant $C^{*}$ such that 
\be\label{est of I_(1)}I_{1}\leq C^{*}\sqrt{t_{k}}.\ee
If $t_{k}$ satisfies 
\be\label{small t_(k)}
t_{k}\leq \frac{1}{16(C^{*})^{2}},\ee
then $|I_{2}|\leq I_{1}\leq \frac{1}{4}$. As a result, $1-2I_{1}\geq \frac{1}{2}$ and 
\[\frac{1+2I_{2}}{1-2I_{1}}=1+\frac{2(I_{1}+I_{2})}{1-2I_{1}}\leq 1+4(I_{1}+I_{2}).\]
Hence by dividing $1-2I_{1}$ from both sides of (\ref{abstract recursive ineq}), we obtain
\be\label{abstract recursive ineq, simp}
M_{k} \leq \big[1+4(I_{1}+I_{2})\big]M_{k-1}+4I_{3}M_{k}^{q}. 
\ee

In the following, by choosing a suitable sequence $\{M_{k}\}_{k\geq 0}$ and obtaining a lower bound $t_{k*}$ for each $t_{k}$, the sum of all $t_{k*}$ becomes a lower bound for $T^{*}$. First, due to the estimate (\ref{est of I_(1)}) again, 
\be\label{bdd for ratio, general}
I_{1}+I_{2}\leq 2I_{1}\leq 2C^{*}\sqrt{t_{k}}.\ee
Next in order to estimate $I_{3}$, we apply (\ref{bdry-time int bdd}) for $\Gamma=\Gamma_{1}$ and $\alpha=\frac{1}{2(n-1)}$, then there exists some constant $C$ such that
\be\label{bdd for I_(3), general}
I_{3}\leq C|\Gamma_{1}|^{\alpha}t_{k}^{1/4}.\ee
Plugging (\ref{bdd for ratio, general}) and (\ref{bdd for I_(3), general}) into (\ref{abstract recursive ineq, simp}) yields
\be\label{recursive ineq, general}
M_{k}\leq (1+C\sqrt{t_{k}})M_{k-1}+C|\Gamma_{1}|^{\alpha}\,t_{k}^{1/4}\,M_{k}^{q}.\ee
Define 
\be\label{def of M_(k), general}
M_{k}=2^{k}M_{0}.\ee
Then 
\[2^{k}M_{0}\leq (1+C\sqrt{t_{k}})\,2^{k-1}M_{0}+C|\Gamma_{1}|^{\alpha}\,t_{k}^{1/4}\,2^{qk}M_{0}^{q}.\]
Subtracting $2^{k-1}M_{0}$ from both sides, we obtain
\[2^{k-1}M_{0}\leq C\sqrt{t_{k}}\,2^{k-1}M_{0}+C|\Gamma_{1}|^{\alpha}\,t_{k}^{1/4}\,2^{qk}M_{0}^{q}.\]
Dividing by $2^{k-1}M_{0}$,
\[1\leq C\sqrt{t_{k}}+C|\Gamma_{1}|^{\alpha}\,t_{k}^{1/4}\,2^{(q-1)k}M_{0}^{q-1}.\]
Thus,
\[\sqrt{t_{k}}+|\Gamma_{1}|^{\alpha}M_{0}^{q-1}2^{(q-1)k}\,t_{k}^{1/4}-\frac{1}{C}\geq 0.\]
Regarding the left hand side of the above inequality to be a quadratic function in $t_{k}^{1/4}$, then $t_{k}^{1/4}$ has to be greater than its positive root, that is
\[t_{k}^{1/4}\geq \frac{1}{2}\bigg(-|\Gamma_{1}|^{\alpha}M_{0}^{q-1}\,2^{(q-1)k}+\sqrt{|\Gamma_{1}|^{2\alpha}M_{0}^{2(q-1)}\,2^{2(q-1)k}+\frac{4}{C}}\,\bigg).\]
Consequently,
\begin{align*}
t_{k}^{1/4} &\geq \frac{2}{C\Big(|\Gamma_{1}|^{\alpha}M_{0}^{q-1}\,2^{(q-1)k}+\sqrt{|\Gamma_{1}|^{2\alpha}M_{0}^{2(q-1)}\,2^{2(q-1)k}+\frac{4}{C}}\Big)}\\
&\geq \frac{1}{C\sqrt{|\Gamma_{1}|^{2\alpha}M_{0}^{2(q-1)}\,2^{2(q-1)k}+\frac{4}{C}}}.
\end{align*}
Hence, there exists $C_{1}^{*}$ such that
\be\label{lower bound for t_(k), cond}
t_{k}\geq \frac{1}{C_{1}^{*}\,\big(|\Gamma_{1}|^{4\alpha}M_{0}^{4(q-1)}\,2^{4(q-1)k}+1\big)}.\ee

As a summary of the above paragraph, by choosing $M_{k}=2^{k}M_{0}$, then (\ref{small t_(k)}) implies (\ref{lower bound for t_(k), cond}). Therefore, 
\[t_{k} \geq \min\bigg\{\frac{1}{16(C^{*})^{2}},\, \frac{1}{C_{1}^{*}\,\big(|\Gamma_{1}|^{4\alpha}M_{0}^{4(q-1)}\,2^{4(q-1)k}+1\big)}\bigg\}.\]
Denoting \[C_{2}^{*}=\min\Big\{\frac{1}{16(C^{*})^{2}}, \frac{1}{C_{1}^{*}}\Big\},\]
then
\be\label{lower bound for t_(k)}
t_{k}\geq \frac{C_{2}^{*}}{|\Gamma_{1}|^{4\alpha}M_{0}^{4(q-1)}\,2^{4(q-1)k}+1}.\ee
Hence,
\begin{align*}
T^{*} =\sum_{k=1}^{\infty}t_{k}&\geq C_{2}^{*}\sum_{k=1}^{\infty}\frac{1}{|\Gamma_{1}|^{4\alpha}M_{0}^{4(q-1)}\,2^{4(q-1)k}+1}\\
&\geq C_{2}^{*}\int_{1}^{\infty}\frac{1}{|\Gamma_{1}|^{4\alpha}M_{0}^{4(q-1)}\,2^{4(q-1)x}+1}\,dx\\
&= \frac{C_{2}^{*}}{4(q-1)\ln(2)}\ln\Bigg(1+\frac{1}{|\Gamma_{1}|^{4\alpha}M_{0}^{4(q-1)}\,2^{4(q-1)}}\Bigg).
\end{align*}
Recalling $\alpha=\frac{1}{2(n-1)}$, (\ref{lower bdd, general domain}) follows.
\end{proof}

\section{Proof of Theorem \ref{Thm, lower bdd, convex}}
\label{Sec, proof for convex}
Define  
\be\label{const on q}
E_{q}=(q-1)^{q-1}/q^q, \quad\forall\, q>1. \ee
By elementary calculus,
\be\label{est on E_q}
\frac{1}{3q}<E_{q}<\min\Big\{ \frac{1}{q}, \frac{1}{(q-1)\,e}\Big\}<1. \ee
The lemma below is a simple generalization of Lemma 3.2 in \cite{YZ1611}.
\begin{lemma}\label{Lemma, criteria for step continue}
For any $q>1$ and $m>0$, write $E_{q}$ as in (\ref{const on q}) and define $g:(m,\infty)\rightarrow \m{R}$ by
\be\label{auxiliary fcn}
g(\lambda)=\frac{\lambda-m}{\lambda^{q}},\quad\forall\,\lam>m.\ee
Then the following two claims hold.
\begin{itemize}
\item[(1)] For any $y\in\big(0, m^{1-q}E_{q}\big]$, there exists unique $\lambda\in\big(m,\frac{q}{q-1}m\big]$ such that $g(\lambda)=y$.
\item[(2)] For any $y>m^{1-q}E_{q}$, there does not exist $\lambda>m$ such that $g(\lambda)=y$.
\end{itemize}
\end{lemma}
\begin{proof}
Since $g$ is strictly increasing on the interval $\big(m,\frac{q}{q-1}m\big]$ and strictly decreasing on the interval $\big[\frac{q}{q-1}m, \infty\big)$, it reaches the maximum at $\lambda=\frac{q}{q-1}m$. Noticing that
\[g\Big(\frac{q}{q-1}m\Big)=m^{1-q}E_{q},\]
then the claims (1) and (2) follow directly. 
\end{proof}

Now we can carry out the main proof in this section.
\begin{proof}[Proof of Theorem \ref{Thm, lower bdd, convex}]
We will demonstrate detailed proof for the case $n\geq 3$, the proof for the case $n=2$ is similar and will be briefly mentioned at the end. In the proof below, $C$ will denote the constants which only depend on $n$ and $\O$, the values of $C$ may be different in different places. But $C^{*}$ will represent a fixed constant which only depends on $n$ and $\O$. Let $M(t)$ be defined as in (\ref{max function at time t}). 

{\bf Step 1.} The first part is exactly the same as the second paragraph in the proof of Theorem \ref{Thm, lower bdd, general domain}, namely we adopt the same notations and the same estimates from (\ref{difference in time, kth step, general}) through (\ref{abstract recursive ineq, simp}). In particular, we make the assumption (\ref{small t_(k)}). 

{\bf Step 2.} In this step, we will find a  constant $t_{*}>0$ and a finite strictly increasing sequence $\{M_{k}\}_{0\leq k\leq L}$ such that $t_{k}\geq t_{*}$ for $1\leq k\leq L$. Then in Step 3, a lower bound for $Lt_{*}$ will be derived. 

Due to the convexity of $\O$, the normal derivative $\dfrac{\p\Phi(x^{k}-y,t_{k}-\tau)}{\p n(y)}$ in (\ref{def of aux int}) is always nonpositive. As a result, 
\be\label{bdd for ratio, convex}
I_{1}+I_{2}=\int_{0}^{t_{k}}\int_{\p\O}\frac{\p\Phi(x^{k}-y,t_{k}-\tau)}{\p n(y)}+\bigg|\frac{\p\Phi(x^{k}-y,t_{k}-\tau)}{\p n(y)}\bigg|\,dS(y)\,d\tau=0.\ee
To estimate $I_{3}$, we apply Lemma \ref{Lemma, bdry-time int bdd, critical, n large} to conclude
\be\label{bdd for I_(3), convex}
I_{3}\leq C|\Gamma_{1}|^{1/(n-1)}\ee
for some constant $C=C(n,\O)$. Hence plugging (\ref{bdd for ratio, convex}) and (\ref{bdd for I_(3), convex}) into (\ref{abstract recursive ineq, simp}), we get
\be\label{recursive ineq, convex}
M_{k}\leq M_{k-1}+C^{*}|\Gamma_{1}|^{1/(n-1)}M_{k}^{q}\ee
for some constant $C^{*}=C^{*}(n,\O)$. As a summary, the argument so far claims that if (\ref{small t_(k)}) holds, then $M_{k}$ will satisfy (\ref{recursive ineq, convex}).

Based on the above observation, if we choose
\be\label{choice of delta_(1), convex}
\delta_{1}=2C^{*}|\Gamma_{1}|^{1/(n-1)}\ee
and define $M_{k}$ to be the solution (if it exists) to
\be\label{choice of M_(k), convex}
\frac{M_{k}-M_{k-1}}{M_{k}^{q}}=\delta_{1}, \ee
then (\ref{small t_(k)}) can not hold since otherwise (\ref{recursive ineq, convex}) will be violated. Consequently $t_{k}>t_{*}$, where
\be\label{def of t_(*), convex}
t_{*}=\frac{1}{16(C^{*})^{2}}. \ee

Due to Lemma \ref{Lemma, criteria for step continue}, the existence of a solution $M_{k}$ to (\ref{choice of M_(k), convex}) is equivalent to the inequality $M_{k-1}^{q-1}\delta_{1}\leq E_{q}$. In addition, as long as such a solution exists, $M_{k}$ can be chosen to satisfy
\[M_{k-1}<M_{k}\leq \frac{q}{q-1}\,M_{k-1}.\] 
Thus, the strategy of constructing $\{M_{k}\}$ is summarized as below. First, define $M_{0}$ and $\delta_{1}$ as in (\ref{initial max}) and (\ref{choice of delta_(1), convex}). Next suppose $M_{k-1}$ has been constructed for some $k\geq 1$, then whether defining $M_{k}$ depends on how large $M_{k-1}$ is.
\begin{itemize}
\item[$\diamond$] If $M_{k-1}^{q-1}\,\delta_{1}\leq E_{q}$, then we define $M_{k}\in \big(M_{k-1}, \frac{q}{q-1}\,M_{k-1}\big]$ to be the solution to (\ref{choice of M_(k), convex}).

\item[$\diamond$] If $M_{k-1}^{q-1}\,\delta_{1}> E_{q}$, then there does not exist $M_{k}>M_{k-1}$ which solves (\ref{choice of M_(k), convex}). So we do not define $M_{k}$ and stop the construction.  
\end{itemize}

According to this construction, if $\{M_{k}\}_{1\leq k\leq L_{0}}$ have been defined, then $T_{k}-T_{k-1}\geq t_{*}$ for any $1\leq k\leq L_{0}$. Therefore, $T_{k}\geq kt_{*}$ for any $1\leq  k\leq L_{0}$. Since $T^{*}$ is finite, $L_{0}\leq T^{*}/t_{*}<\infty$, which means the cardinality of $\{M_{k}\}$ has to be finite (actually this fact can also be justified by analysing the construction directly, see Lemma \ref{Lemma, lower bdd of steps, basic}). So we can assume the constructed sequence is $\{M_{k}\}_{0\leq k\leq L}$ for some finite $L$. 

{\bf Step 3.} By Lemma \ref{Lemma, lower bdd of steps, basic}, 
\[L>\frac{1}{10(q-1)}\Big(\frac{1}{M_{0}^{q-1}\delta_{1}}-3q\Big).\]
To obtain an effective lower bound, $\frac{1}{M_{0}^{q-1}\delta_{1}}$ should be greater than $3q$. If requiring 
\be\label{small surface area, convex}
\frac{1}{M_{0}^{q-1}\delta_{1}}\geq 6q,\ee
then 
\[L>\frac{1}{20(q-1)M_{0}^{q-1}\delta_{1}}=\frac{1}{40C^{*}(q-1)M_{0}^{q-1}|\Gamma_{1}|^{1/(n-1)}}.\]
Denote $Y=M_{0}^{q-1}|\Gamma_{1}|^{1/(n-1)}$. Then
\[T^{*}\geq Lt_{*}>\frac{C}{(q-1)Y}\]
for some constant $C$. Finally, noticing that (\ref{small surface area, convex}) is equivalent to 
\[Y\leq \frac{1}{12C^{*}q},\]
the proof for the case $n\geq 3$ is finished by setting $Y_{0}=1/(12C^{*})$.

When $n=2$, the process is almost identical as the above except two differences. First, Lemma \ref{Lemma, bdry-time int bdd, critical, n=2} will be applied instead of Lemma \ref{Lemma, bdry-time int bdd, critical, n large}, so the term $|\Gamma_{1}|^{1/(n-1)}$ in the above proof needs to be replaced by $|\Gamma_{1}|\ln\big(\frac{1}{|\Gamma_{1}|}+1\big)$. Secondly, due to the restriction $t\leq 1$ in Lemma \ref{Lemma, bdry-time int bdd, critical, n=2}, $t_{k}$ should satisfy both $t_{k}\leq 1$ and (\ref{small t_(k)}) in order to justify (\ref{recursive ineq, convex}). Consequently the choice of $t_{*}$ will be
\be\label{def of t_(*), convex, n=2}
t_{*}=\min\Big\{\frac{1}{16(C^{*})^{2}},\, 1\Big\}\ee
instead of (\ref{def of t_(*), convex}). Fortunately, this additional requirement will not bring major changes to the proof. Actually, without loss of generality, we can choose $C^{*}$ to be larger than $1/4$, which makes $\frac{1}{16(C^{*})^{2}}\leq 1$. As a result, (\ref{def of t_(*), convex, n=2}) coincides with (\ref{def of t_(*), convex}). Then the rest of the proof is the same.
\end{proof}

The following lemma has been applied in the proof of Theorem \ref{Thm, lower bdd, convex} and will be used again in the proof of Theorem \ref{Thm, lower bdd, locally convex}, so we state it separately for convenience. It is a generalization of Lemma 3.3 in \cite{YZ1611}, but its statement and proof are much simpler. 
\begin{lemma}\label{Lemma, lower bdd of steps, basic}
Given $q>1$, $M_{0}>0$ and $\delta_{1}>0$, denote $E_{q}$ as  (\ref{const on q}) and construct a (finite) sequence $\{M_{k}\}_{k\geq 0}$ inductively as follows. Suppose $M_{k-1}$ has been constructed for some $k\geq 1$, then based on Lemma \ref{Lemma, criteria for step continue}, whether defining $M_{k}$ depends on how large $M_{k-1}$ is.
\begin{itemize}
\item[$\diamond$] If $M_{k-1}^{q-1}\,\delta_{1}\leq E_{q}$, then we define $M_{k}\in \big(M_{k-1}, \frac{q}{q-1}\,M_{k-1}\big]$ to be the solution to 
\be\label{def of M_(k), lemma}
\frac{M_{k}-M_{k-1}}{M_{k}^{q}}=\delta_{1}. \ee

\item[$\diamond$] If $M_{k-1}^{q-1}\,\delta_{1}>E_{q}$, then there does not exist $M_{k}>M_{k-1}$ which solves (\ref{def of M_(k), lemma}). So we do not define $M_{k}$ and stop the construction.  
\end{itemize}
We claim this construction stops in finite steps and if the last term is denoted as $M_{L}$, then   
\be\label{lower bdd of steps, basic} 
L>\frac{1}{10(q-1)}\Big(\frac{1}{M_{0}^{q-1}\delta_{1}}-3q\Big).\ee
\end{lemma}

\begin{proof} 
First, we will show the construction has to stop in finite steps. In fact, it follows from (\ref{def of M_(k), lemma}) that the sequence $\{M_{k}\}$ is strictly increasing and 
\[M_{k}= M_{k-1}+M_{k}^{q}\delta_{1}\geq \big(1+M_{0}^{q-1}\delta_{1}\big)M_{k-1}.\]
As a result, 
\[M_{k}\geq \big(1+M_{0}^{q-1}\delta_{1}\big)^{k}M_{0}.\]
Thus $M_{k}^{q-1}$ will exceed $E_{q}/\delta_{1}$ when $k$ is sufficiently large, which forces the construction to stop. 

Next suppose the constructed sequence is $\{M_{k}\}_{0\leq k\leq L}$. The lower bound (\ref{lower bdd of steps, basic}) for $L$ will be justified below.

{\bf Case 1.} $M_{0}^{q-1}\delta_{1}>E_{q}$. In this case, it follows from (\ref{est on E_q}) that
\[\frac{1}{M_{0}^{q-1}\delta_{1}}<\frac{1}{E_{q}}<3q.\]
Thus (\ref{lower bdd of steps, basic}) holds automatically since the right hand side of (\ref{lower bdd of steps, basic}) is negative. 

{\bf Case 2.} $M_{0}^{q-1}\delta_{1}\leq E_{q}$. In this case, it is evident from the construction that $L\geq 1$. Moreover, 
\[M_{L-1}^{q-1}\delta_{1}\leq E_{q} \quad \text{and}\quad M_{L}^{q-1} \delta_{1}> E_{q}.\]
According to the recursive relation (\ref{def of M_(k), lemma}), 
\[M_{k-1}=M_{k}\big(1-M_{k}^{q-1} \delta_{1}\big).\] 
Raising both sides to the power $q-1$ and multiplying by $\delta_{1}$,
\[M_{k-1}^{q-1}\delta_{1}=M_{k}^{q-1}\big(1-M_{k}^{q-1} \delta_{1}\big)^{q-1}\delta_{1}.\]
Let $x_{k}=M_{k}^{q-1}\delta_{1}$. Then 
\be\label{iteration for x_k}
x_{k-1}=x_{k}\,(1-x_{k})^{q-1}, \quad\forall\, 1\leq k\leq L.\ee
Moreover, $$x_0=M_{0}^{q-1}\delta_{1},\quad x_{L-1}\leq E_{q} \quad\text{and}\quad x_{L}>E_{q}.$$
Noticing that $M_{L}\leq \frac{q}{q-1}M_{L-1}$, so
\[x_{L}=\bigg(\frac{M_{L}}{M_{L-1}}\bigg)^{q-1}x_{L-1}\leq \Big(\frac{q}{q-1}\Big)^{q-1}E_{q}=\frac{1}{q}.\]

Since the right hand side of (\ref{iteration for x_k}) is nonlinear in $x_{k}$, it seems impossible to express $x_{k}$ as an explicit formula in terms of $x_{k-1}$. This motivates us to consider the ``reversed'' relation of (\ref{iteration for x_k}), namely a new sequence $\{y_{k}\}_{0\leq k\leq L}$ defined in the following way: $y_{0}\triangleq \min\{1/2, E_{q}\}$ and
\be\label{iteration for y_k}
y_{k}\triangleq y_{k-1}(1-y_{k-1})^{q-1}, \quad \forall\, 1\leq k\leq L.\ee
To analyse the sequence $\{y_{k}\}$, we define $h:(0,1)\rightarrow\m{R}$ by 
\[h(t)=t\,(1-t)^{q-1}\]
so that $y_{k}=h(y_{k-1})$ for $1\leq k\leq L$. It is easy to see that $h$ is strictly increasing on $(0,1/q]$ and strictly decreasing on $[1/q,1)$. Noticing $0<y_0<x_{L}\leq 1/q$, so 
$$y_{1}=h(y_0)<h(x_{L})=x_{L-1}.$$
Keep doing this, we get $y_{k}<x_{L-k}$ for any $0\leq k\leq L$. In particular, $y_{L}<x_{0}=M_{0}^{q-1}\delta_{1}$. 

Since $\{y_{k}\}$ is a decreasing positive sequence and $y_{0}\leq 1/2$, then $y_{k}\leq 1/2$ for any $0\leq k\leq L$. As a result, it follows from (\ref{iteration for y_k}) and the mean value theorem that for any $1\leq k\leq L$,
\be\label{quadratic ineq for y_k}
y_{k}\geq y_{k-1}\big[1-2(q-1)y_{k-1}\big].\ee
Recalling (\ref{est on E_q}) again, 
$$y_{k-1}\leq y_{0}\leq E_{q}< \frac{1}{(q-1)e},$$
so \[1-2(q-1)y_{k-1}>1-\frac{2}{e}>\frac{1}{5}.\]
Hence, taking the reciprocal in (\ref{quadratic ineq for y_k}) yields 
\begin{align}
\frac{1}{y_{k}} &\leq \frac{1}{y_{k-1}\big[1-2(q-1)y_{k-1}\big]} \notag\\
& =\frac{1}{y_{k-1}}+\frac{2(q-1)}{1-2(q-1)y_{k-1}} \notag\\
&< \frac{1}{y_{k-1}}+10(q-1). \label{iteration ineq}
\end{align}
Summing up (\ref{iteration ineq}) for $k$ from $1$ to $L$, then
\begin{align}\label{ineq between last and first for y_k}
\frac{1}{y_{L}} <\frac{1}{y_{0}}+10(q-1)L.
\end{align}
Since $y_{L}<M_{0}^{q-1}\delta_{1}$ and 
\[y_{0}=\min\Big\{\frac{1}{2}, E_{q}\Big\}>\frac{1}{3q},\]
it follows from (\ref{ineq between last and first for y_k}) that \[\frac{1}{M_{0}^{q-1}\delta_{1}}<3q+10(q-1)L.\]
Thus,  
\[L>\frac{1}{10(q-1)}\Big(\frac{1}{M_{0}^{q-1}\delta_{1}}-3q\Big).\]
\end{proof}

\section{Proof of Theorem \ref{Thm, lower bdd, locally convex}}
\label{Sec, proof for locally convex}
\begin{proof}[Proof of Theorem \ref{Thm, lower bdd, locally convex}]

We will demonstrate detailed proof for the case $n\geq 3$, the proof for the case $n=2$ is similar and will be briefly mentioned at the end. In the proof below, $C$ and $C_{i}(i\geq 1)$ will denote the constants which only depend on $n$, $\O$ and $d$, the values of $C$ and $C_{i}$ may be different in different places. But $C^{*}$ and $C_{i}^{*}(i\geq 1)$ will represent fixed constants which only depend on $n$, $\O$ and $d$. Let $M(t)$ be defined as in (\ref{max function at time t}).

{\bf Step 1.} The first part is exactly the same as the second paragraph in the proof of Theorem \ref{Thm, lower bdd, general domain}, namely we adopt the same notations and the same estimates from (\ref{difference in time, kth step, general}) through (\ref{abstract recursive ineq, simp}). In particular, we make the assumption (\ref{small t_(k)}). 

{\bf Step 2.} In this step, we will find a  constant $t_{*}>0$ and a finite strictly increasing sequence $\{M_{k}\}_{0\leq k\leq L}$ such that $t_{k}\geq t_{*}$ for $1\leq k\leq L$. Then in Step 3, a lower bound for $Lt_{*}$ will be derived. 

Due to the local convexity near $\Gamma_{1}$, it follows from (\ref{identity, bdry}) and Corollary \ref{Cor, perturbation of id} that 
\[I_{1}+I_{2}\leq C\,t_{k}\exp\Big(-\frac{d^{2}}{8t_{k}}\Big).\]
Because of the assumption (\ref{small t_(k)}), the above inequality implies 
\be\label{bdd for ratio, local convex}
I_{1}+I_{2}\leq C\exp\Big(-\frac{d^{2}}{8t_{k}}\Big).\ee
On the other hand, we estimate $I_{3}$ in the same way as (\ref{bdd for I_(3), convex}). Now plugging (\ref{bdd for ratio, local convex}) and (\ref{bdd for I_(3), convex}) into (\ref{abstract recursive ineq, simp}), then
\be\label{recursive ineq, local convex}
M_{k}\leq \Big[1+C_{1}^{*}\exp\Big(-\frac{d^{2}}{8t_{k}}\Big)\Big]M_{k-1}+C_{2}^{*}|\Gamma_{1}|^{1/(n-1)}M_{k}^{q},\ee
for two constants $C_{1}^{*}$ and $C_{2}^{*}$. Next if $t_{k}$ is so small that
\be\label{cond on time step, kth}
\exp\Big(-\frac{d^{2}}{8t_{k}}\Big)\leq \frac{1}{2C_{1}^{*}}\,\frac{M_{k}-M_{k-1}}{M_{k-1}},\ee
which is equivalent to 
\[M_{k}-\Big[1+C_{1}^{*}\exp\Big(-\frac{d^{2}}{8t_{k}}\Big)\Big]M_{k-1}\geq \frac{1}{2}(M_{k}-M_{k-1}).\]
Then it follows from (\ref{recursive ineq, local convex}) that
\be\label{recursive ineq, locally convex}
\frac{M_{k}-M_{k-1}}{M_{k}^{q}}\leq 2C_{2}^{*}|\Gamma_{1}|^{1/(n-1)}.\ee
As a summary, the argument so far claims if both (\ref{small t_(k)}) and (\ref{cond on time step, kth}) hold, then $M_{k}$ will satisfy (\ref{recursive ineq, locally convex}).

Based on this observation, if we choose
\be\label{choice of delta_(1), local convex}
\delta_{1}=4C_{2}^{*}|\Gamma_{1}|^{1/(n-1)}\ee
and define $M_{k}$ to be the solution (if it exists) to
\be\label{choice of M_(k), local convex}
\frac{M_{k}-M_{k-1}}{M_{k}^{q}}=\delta_{1}, \ee
then either (\ref{small t_(k)}) or (\ref{cond on time step, kth}) can not hold since otherwise (\ref{recursive ineq, locally convex}) will be violated. The invalidity of (\ref{small t_(k)}) means 
\be\label{lower bdd 1 for t_(k)}t_{k}>\frac{1}{16(C^{*})^{2}}.\ee 
On the other hand, due to (\ref{choice of M_(k), local convex}), the failure of (\ref{cond on time step, kth}) implies
\be\label{t_(k) not small}
\exp\Big(-\frac{d^{2}}{8t_{k}}\Big) >\frac{1}{2C_{1}^{*}}\,\frac{M_{k}-M_{k-1}}{M_{k-1}}=\frac{M_{k}^{q}\delta_{1}}{2C_{1}^{*}M_{k-1}}\geq \frac{M_{0}^{q-1}\delta_{1}}{2C_{1}^{*}}.\ee
If 
\be\label{small surface area, first}
M_{0}^{q-1}\delta_{1}\leq C_{1}^{*},\ee
then the right hand side of (\ref{t_(k) not small}) is smaller than 1. Therefore, (\ref{t_(k) not small}) is equivalent to 
\be\label{lower bdd 2 for t_(k)} t_{k}>\frac{d^{2}}{8}\bigg[\ln\Big(\frac{2C_{1}^{*}}{M_{0}^{q-1}\delta_{1}}\Big)\bigg]^{-1}.\ee 
In summary, if (\ref{small surface area, first}) holds, then it follows from (\ref{lower bdd 1 for t_(k)}) and (\ref{lower bdd 2 for t_(k)}) that $t_{k}\geq t_{*}$, where 
\be\label{def of t_(*), locally convex}
t_{*}=\min\bigg\{\frac{1}{16(C^{*})^{2}},\, \frac{d^{2}}{8}\bigg[\ln\Big(\frac{2C_{1}^{*}}{M_{0}^{q-1}\delta_{1}}\Big)\bigg]^{-1}\bigg\}.\ee

Due to Lemma \ref{Lemma, criteria for step continue}, the existence of a solution $M_{k}$ to (\ref{choice of M_(k), local convex}) is equivalent to the inequality $M_{k-1}^{q-1}\delta_{1}\leq E_{q}$. In addition, as long as such a solution exists, $M_{k}$ can be chosen to satisfy
\[M_{k-1}<M_{k}\leq \frac{q}{q-1}\,M_{k-1}.\] 
Thus, the strategy of constructing $\{M_{k}\}$ is summarized as following. First, define $M_{0}$ and $\delta_{1}$ as in (\ref{initial max}) and (\ref{choice of delta_(1), local convex}). Next suppose $M_{k-1}$ has been constructed for some $k\geq 1$, then whether defining $M_{k}$ depends on how large $M_{k-1}$ is.
\begin{itemize}
\item[$\diamond$] If $M_{k-1}^{q-1}\,\delta_{1}\leq E_{q}$, then we define $M_{k}\in \big(M_{k-1}, \frac{q}{q-1}\,M_{k-1}\big]$ to be the solution to (\ref{choice of M_(k), local convex}).

\item[$\diamond$] If $M_{k-1}^{q-1}\,\delta_{1}> E_{q}$, then there does not exist $M_{k}>M_{k-1}$ which solves (\ref{choice of M_(k), local convex}). So we do not define $M_{k}$ and stop the construction.  
\end{itemize}

According to this construction, if $\{M_{k}\}_{1\leq k\leq L_{0}}$ have been defined, then $T_{k}-T_{k-1}\geq t_{*}$ for any $1\leq k\leq L_{0}$. Therefore, $T_{k}\geq kt_{*}$ for any $1\leq  k\leq L_{0}$. Since $T^{*}$ is finite, $L_{0}\leq T^{*}/t_{*}<\infty$, which means the cardinality of $\{M_{k}\}$ has to be finite (actually this fact can also be justified by analysing the construction directly, see Lemma \ref{Lemma, lower bdd of steps, basic}). So we can assume the constructed sequence is $\{M_{k}\}_{0\leq k\leq L}$ for some finite $L$.

{\bf Step 3.} By Lemma \ref{Lemma, lower bdd of steps, basic}, 
\[L>\frac{1}{10(q-1)}\Big(\frac{1}{M_{0}^{q-1}\delta_{1}}-3q\Big).\]
If 
\be\label{small surface area, second}
M_{0}^{q-1}\delta_{1}\leq \frac{1}{6q},\ee
then 
\[L\geq \frac{1}{20(q-1)M_{0}^{q-1}\delta_{1}}.\]
Combining the assumptions (\ref{small surface area, first}) and  (\ref{small surface area, second}), if 
\be\label{small surface area}
M_{0}^{q-1}\delta_{1}\leq \min\Big\{C_{1}^{*}, \frac{1}{6q}\Big\}, \ee
then 
\be\label{lower bdd for T^(*), rough}
T^{*}\geq Lt_{*}\geq \frac{1}{20(q-1)M_{0}^{q-1}\delta_{1}}\min\bigg\{\frac{1}{16(C^{*})^{2}},\, \frac{d^{2}}{8}\bigg[\ln\Big(\frac{2C_{1}^{*}}{M_{0}^{q-1}\delta_{1}}\Big)\bigg]^{-1}\bigg\}.\ee
Denote
\[Y=M_{0}^{q-1}|\Gamma_{1}|^{1/(n-1)}.\]
Recalling $\delta_{1}=4C_{2}^{*}|\Gamma_{1}|^{1/(n-1)}$, we can rewrite (\ref{small surface area}) and (\ref{lower bdd for T^(*), rough}) as 
\be\label{small surface area'}
Y\leq \min\Big\{\frac{C_{1}^{*}}{4C_{2}^{*}},\, \frac{1}{24C_{2}^{*}q}\Big\}
\tag{\ref{small surface area}$'$}\ee
and
\be\label{lower bdd for T^(*), rough'}
T^{*}\geq \frac{1}{80C_{2}^{*}(q-1)Y}\min\bigg\{\frac{1}{16(C^{*})^{2}},\, \frac{d^{2}}{8}\bigg[\ln\Big(\frac{C_{1}^{*}}{2C_{2}^{*}Y}\Big)\bigg]^{-1}\bigg\}.
\tag{\ref{lower bdd for T^(*), rough}$'$}\ee
In order to simplify (\ref{lower bdd for T^(*), rough'}), if
\be\label{additional bound on Y}
Y\leq \min\Big\{\frac{2C_{2}^{*}}{C_{1}^{*}},\, \frac{C_{1}^{*}}{2C_{2}^{*}}\exp\big[-2d^{2}(C^{*})^{2}\big]\Big\},\ee
then
\begin{align*}
2d^{2}(C^{*})^{2}\leq\ln\Big(\frac{C_{1}^{*}}{2C_{2}^{*}Y}\Big) &\leq 2\ln\Big(\frac{1}{Y}\Big)=2|\ln Y|.
\end{align*}
Taking reciprocal of the above inequality and multiplying by $d^{2}/8$, we obtain
\[\frac{d^{2}}{16|\ln Y|}\leq\frac{d^{2}}{8}\bigg[\ln\Big(\frac{C_{1}^{*}}{2C_{2}^{*}Y}\Big)\bigg]^{-1}\leq \frac{1}{16(C^{*})^{2}}.\]
Therefore, (\ref{lower bdd for T^(*), rough'}) yields
\[T^{*}\geq \frac{1}{80C_{2}^{*}(q-1)Y}\,\frac{d^{2}}{16|\ln Y|}=\frac{C}{(q-1)Y|\ln Y|}.\]
Combining the assumptions (\ref{small surface area'}) and (\ref{additional bound on Y}) together, it suffices to require $Y\leq Y_{0}/q$, where
\[Y_{0}=\min\Big\{\frac{1}{24C_{2}^{*}},\,\frac{2C_{2}^{*}}{C_{1}^{*}},\,\frac{C_{1}^{*}}{4C_{2}^{*}}\exp\big[-2d^{2}(C^{*})^{2}\big]\Big\}\] 
is a constant which only depends on $n$, $\O$ and $d$. Hence, we finish the proof when $n\geq 3$.

When $n=2$, we can argue in the same way as the last paragraph of the proof for Theorem \ref{Thm, lower bdd, convex} to justify the conclusion. 
\end{proof}

\section*{Acknowledgements}
The authors appreciate Willie Wong's suggestions which simplify several proofs in this paper and make the ideas clearer. The authors also thank the referee for the careful reading and helpful suggestions.
\bigskip





\bibliographystyle{plain}
\bibliography{References}

\end{document}